\documentclass[11pt]{amsart}

\usepackage{epigamath}

\usepackage[english]{babel}

\numberwithin{equation}{section}

\usepackage[shortlabels]{enumitem}
\setlist[enumerate,1]{label={\rm(\arabic*)}, ref={\rm\arabic*}} 

\usepackage{amsmath,amssymb,amsthm}
\usepackage{mathtools} 
\usepackage{stackengine}
\usepackage{array}
\usepackage{calligra,mathrsfs} 

\usepackage{tikz}
\usetikzlibrary{decorations.pathmorphing}
\usepackage{tikz-cd}

\usepackage[capitalize]{cleveref}
\crefformat{equation}{\normalfont(#2#1#3)}

\newtheorem{thm}{Theorem}[section]
\newtheorem{Theorem}[thm]{Theorem}
\newtheorem{Lemma}[thm]{Lemma}
\newtheorem{Proposition}[thm]{Proposition}
\newtheorem{Corollary}[thm]{Corollary}
\newtheorem{Claim}[thm]{Claim}

\theoremstyle{definition}
\newtheorem{Definition}[thm]{Definition}
\newtheorem{Setup}[thm]{Setup}

\theoremstyle{remark}
\newtheorem{Example}[thm]{Example}
\newtheorem{Remark}[thm]{Remark}

\newlist{propenum}{enumerate}{1}
\setlist[propenum]{label={\rm(\roman*)}, ref=\theProposition.(\roman*)}
\crefalias{propenumi}{Proposition}

\newlist{defenum}{enumerate}{1}
\setlist[defenum]{label=(\alph*), ref=\theDefinition.(\alph*)}
\crefalias{defenumi}{Definition}

\newcommand{\lra}{\longrightarrow}
\newcommand*{\lowsim}{\vbox to 0pt{\vss\hbox{$\scriptstyle\sim$}\vskip-2pt}}
\newcommand*{\isomarrow}{\xrightarrow{\lowsim}}
\newcommand*{\longisomarrow}{\stackrel{\lowsim}{\lra}}
\newcommand*{\longhookrightarrow}{\lhook\joinrel\longrightarrow}

\makeatletter
\newcommand{\xrightrightarrows}[2][]{\mathrel{
 \raise.40ex\hbox{$
       \ext@arrow 3095\rightarrowfill@{\phantom{#1}}{#2}$}
 \setbox0=\hbox{$\ext@arrow 0359\rightarrowfill@{#1}{\phantom{#2}}$}
 \kern-\wd0 \lower.40ex\box0}}
\newcommand{\longrightrightarrows}{\xrightrightarrows{\hphantom{0pt}}}

\newcommand{\wh}{\widehat}
\newcommand{\B}{\mathbb{B}}
\newcommand{\C}{\mathbb{C}}
\newcommand{\G}{\mathbb{G}}
\newcommand{\N}{\mathbb{N}}
\renewcommand{\O}{\mathcal{O}}
\newcommand{\Q}{\mathbb{Q}}

\newcommand{\m}{\mathfrak{m}}
\newcommand{\mg}{\mathfrak{g}}

\newcommand{\im}{\operatorname{im}}

\newcommand{\End}{\operatorname{End}}
\newcommand{\GL}{\operatorname{GL}}
\newcommand{\Perf}{\operatorname{Perf}}
\newcommand{\Spec}{\operatorname{Spec}}
\newcommand{\Spa}{\operatorname{Spa}}

\newcommand{\id}{{\operatorname{id}}}
\newcommand{\cts}{{\operatorname{cts}}}
\newcommand{\an}{{\mathrm{an}}}
\newcommand{\et}{{\operatorname{\acute{e}t}}}
\newcommand{\proet}{{\operatorname{pro\acute{e}t}}}
\newcommand{\qproet}{{\operatorname{qpro\acute{e}t}}}

\newcommand{\wt}{\widetilde}
\renewcommand{\lim}{\varprojlim}
\newcommand{\cH}{{\ifmmode \check{H}\else{\v{C}ech}\fi}}
\newcommand{\tf}{[\tfrac{1}{p}]}
\newcommand{\aeq}{\stackrel{a}{=}}
\newcommand{\ad}{\mathrm{ad}}

\DeclareMathOperator{\HOM}{\mathscr{H}\hspace{-0.3em}\mathrm{\textit{om}}}

\newcommand{\Ext}{\operatorname{Ext}}
\newcommand{\coim}{\operatorname{coim}}
\newcommand{\BCH}{\mathrm{BCH}}
\newcommand{\Rep}{\mathrm{Rep}}
\newcommand{\Spf}{\operatorname{Spf}}
\newcommand{\mG}{\mathcal{G}}
\newcommand{\etqcqs}{{\operatorname{\acute{e}t,qcqs}}}
\newcommand{\Sous}{\mathrm{Sous}}
\newcommand{\Dmd}{\mathrm{Dmd}}

\newcommand{\supth}[1]{\ensuremath{#1^{\mathrm{th}}}}

\EpigaVolumeYear{10}{2026} \EpigaArticleNr{2} \ReceivedOn{June 24, 2024}
\InFinalFormOn{August 8, 2025}
\AcceptedOn{September 8, 2025}

\title{$\boldsymbol{G}$-torsors on perfectoid spaces}
\titlemark{$\boldsymbol{G}$-torsors on perfectoid spaces}

\author{Ben Heuer}
\address{Institute for Algebra, Number Theory and Discrete Mathematics, Leibniz University Hannover, Welfengarten 1, 30167 Hannover, Germany} 
\email{ben.heuer@math.uni-hannover.de}

\authormark{B.~Heuer}

\AbstractInEnglish{For any rigid analytic group variety $G$ over a non-Archimedean field  $K$ over $\Q_p$, we study $G$-torsors on adic spaces over $K$ in the $v$-topology. Our main result is that on perfectoid spaces, $G$-torsors in the \'etale and $v$-topology are equivalent. This generalises the known cases of $G=\G_a$ and $G=\GL_n$ due to Scholze and Kedlaya--Liu.
		
  On a general adic space $X$ over $K$, where there can be more $v$-topological $G$-torsors than \'etale ones, we show that for any open subgroup $U\subseteq G$, any \mbox{$G$-torsor} on $X_v$ admits a reduction of structure group to $U$ \'etale-locally on $X$. This has applications in the context of the $p$-adic Simpson correspondence: for example, we use it to show that on any adic space, generalised $\Q_p$-representations are equivalent to $v$-vector bundles.}



\acknowledgement{This work was funded by Deutsche Forschungsgemeinschaft (DFG, German Research Foundation) under Germany's Excellence Strategy -- EXC-2047/1 -- 390685813. The author was supported by DFG via the Leibniz-Preis of Peter Scholze.}

\begin{document}



\maketitle

\begin{prelims}

\DisplayAbstractInEnglish

\bigskip

\DisplayKeyWords

\medskip

\DisplayMSCclass

\end{prelims}


\newpage

\setcounter{tocdepth}{1}

\tableofcontents

	
\section{Introduction}
	Let $G$ be a smooth group scheme over a field $K$, and let $X$ be a $K$-scheme. Then by a well-known theorem of Grothendieck \cite[Th\'eor\`eme~11.7, Remark~11.8]{Grothendieck_BrauerIII}, the functor
	\[ \left\{ \text{$G$-torsors on $X_{\et}$}\right\}\longisomarrow \left\{\text{$G$-torsors on $X_{\mathrm{fppf}}$}\right\}\]
	is an equivalence of categories. This article studies a similar functor in $p$-adic geometry. 
	
	Let $K$ be a non-Archimedean field over $\Q_p$, and let $X$ be an adic space over $K$, considered as a diamond in the sense of Scholze \cite[Section~15]{etale-cohomology-of-diamonds}.
        Then $X$ has an \'etale topology and a much finer $v$-topology; see \cite[Section~14]{etale-cohomology-of-diamonds}.
	Let $G$ be any rigid analytic group variety over $K$, not necessarily commutative, considered as a diamond. In this article, we study the functor
	\[ \left\{ \text{$G$-torsors on $X_{\et}$}\right\}\longhookrightarrow \left\{\text{$G$-torsors on $X_{v}$}\right\},\]
	where a $G$-torsor is a sheaf with a $G$-action that is locally isomorphic to $G$ (see \Cref{s:G-torsors}). Let
	$\nu\colon X_v\to X_{\et}$ be the natural morphism of sites; then this functor is fully faithful under the mild technical assumption that $\nu_\ast \mathbb G_a=\G_a$ (\textit{e.g.}~by
        results of Scholze and Kedlaya--Liu, this holds when $X$ is a perfectoid space or a semi-normal rigid space); see \Cref{p:fully-faithful-G-torsors-et-to-v}. The question whether the functor is essentially surjective can be phrased in terms of $\nu$
	by asking whether $R^1\nu_{\ast}G$  vanishes.
	
	\medskip
	
	In this $p$-adic situation, the  analogy to Grothendieck's theorem works best when $X$ is a perfectoid space: indeed,
        in this case, torsors under $G=\G_a$ on $X_{\et}$ and $X_{v}$ agree; namely, Scholze proves that  $R\nu_{\ast}\G_a=\G_a$; see \cite[Proposition 8.8]{etale-cohomology-of-diamonds}.
	By a theorem of Kedlaya--Liu,  vector bundles on $X_{\et}$ and $X_v$ also agree; see \cite[Theorem~3.5.8]{KedlayaLiu-II} (\textit{cf.}
        \cite[Lemma 17.1.8]{ScholzeBerkeleyLectureNotes}), which is the case of $G=\GL_n$. This implies the case of linear algebraic $G$ by the Tannakian formalism; see \cite[Section~19.5]{ScholzeBerkeleyLectureNotes}.
	Our main result is the following generalisation of all of these statements.
	
	\begin{Theorem}\label{t:G-torsors-on-perfectoid-spaces}
		Let $X$ be a perfectoid space, and let $G$ be a rigid group over $K$. The functor
			\[ \left\{ \text{$G$-torsors on $X_{\et}$}\right\}\longisomarrow \left\{\text{$G$-torsors on $X_{v}$}\right\}\]
			is an equivalence of categories. If\, $G$ is commutative, then we more generally have $R\nu_{\ast}G=G$.
	\end{Theorem}
        
	Our method is different to that of the aforementioned works, and also gives a new proof of Kedlaya--Liu's theorem.
	One interesting new case is $G=\GL_n(\O^+)$, which says that the categories of finite locally free $\O^+$-modules on $X_{\et}$ and $X_v$ agree. 
	
	As another application, recall that for any locally Noetherian adic space $X$ over $\Q_p$ (for example a rigid space), we have a hierarchy of topologies
	\[ X_v\lra X_\qproet\lra X_{\proet}\lra X_\et.\]
	Here $X_\qproet$ is the quasi-pro-\'etale site, which makes sense without the assumption that $X$ is locally Noetherian.
	Since the three topologies on the left are locally perfectoid, we deduce the following.

	\begin{Corollary}
		Let $X$ be any adic space over $K$; then the categories of\, $G$-torsors on $X_{v}$ and  $X_{\qproet}$ $($and $X_\proet$ if\,~$X$ is locally Noetherian$)$ are equivalent. 
	\end{Corollary}
	
\subsection{$\boldsymbol{G}$-torsors on rigid spaces}
Now let $X$ be a rigid space. Then it is known that there are in general many more \mbox{$v$-topological} $G$-torsors than there are \'etale ones: for example, if $K$ is a complete algebraically closed extension of $\Q_p$ and $X$ is smooth over $K$, then already for $G=\G_a$, a key step in Scholze's construction of the Hodge--Tate spectral sequence is a canonical isomorphism on $X_\et$
\[R^1\nu_{\ast}\G_a=\Omega_X(-1), \]
where $\Omega_X(-1)$ is a Tate twist of the K\"ahler differentials on $X$; see  
\cite[Proposition~3.23]{Scholze2012Survey}.

It is also know for $G=\GL_n$ that \'etale vector bundles and $v$-vector bundles on $X$ are not the same; see \textit{e.g.}~\cite{heuer-v_lb_rigid} for the case of $G=\G_m$. The difference is closely related to the $p$-adic Simpson correspondence: as we explain in \cref{s:generalised-representations}, there is an equivalence
\[ \{ \text{finite locally free $\O^+$-modules on $X_{v}$}\}\longisomarrow \{\text{generalised representations on $X$}\},\]
where, following Faltings \cite{Faltings_SimpsonI}, generalised representations are compatible systems $(V_n)_{n\in \N}$ of  finite locally free  $\O^+/p^n$-modules $V_n$ on $X_{\et}$. The equivalence is similar in spirit to the equivalence between lisse $\Q_l$-sheaves on a scheme and locally free $\Q_l$-modules on the pro-\'etale site in the sense of Bhatt--Scholze \cite[Section~1]{bhatt-scholze-proetale}. The proof hinges on the following. 

\begin{Proposition}
	We have an isomorphism $R\nu_{\ast}(\O^+/p^n)=\O^+/p^n$ $($already before passing to the almost category$)$ and $R^1\nu_{\ast}\GL_n(\O^+/p^n)=1$. In particular, we have an equivalence
\[\nu^{\ast}\colon \Big\{ \begin{array}{c}\text{finite locally free}\\\text{$\O^+/p^n$-modules on $X_\et$}\end{array}\Big\}\longisomarrow \Big\{ \begin{array}{c}\text{finite locally free}\\\text{$\O^+/p^n$-modules on $X_v$}\end{array}\Big\}.\]
\end{Proposition}

One reason why we are interested in $v$-topological $G$-torsors under general rigid groups $G$ is for generalisations of the $p$-adic Simpson correspondence to more general non-abelian coefficients, as explored in \cite{heuer-sheafified-paCS}: these relate $v$-topological $G$-bundles to $G$-Higgs bundles. The relevance of \cref{t:G-torsors-on-perfectoid-spaces} in this context is that it shows that $v$-topological $G$-bundles are ``locally small''; namely, they admit reductions of structure groups to small open subgroups.

\subsection{Reduction of structure group}
The technical heart of this article is the study of sheaves $F$ that ``commute with tilde-limits''. 

\begin{Definition}
	A sheaf $F$ on the big \'etale site of sousperfectoid spaces over $K$ is said to satisfy the \emph{approximation property} if for any affinoid perfectoid tilde-limit $X\sim \varprojlim_{i\in I} X_i$ of affinoid spaces $X_i$ such that $\varinjlim_{i\in I} \O(X_i)\to \O(X)$ has dense image, $F(X)=\varinjlim_{i\in I} F(X_i)$.
\end{Definition}

Examples of such sheaves include $\O^+/p^n$ and thus also $\GL_n(\O^+/p^n)$.  We show the following. 

\begin{Theorem}
		Let $F$ be a sheaf of groups satisfying the approximation property. Then $F$ is already a $v$-sheaf and 
			$R^1\nu_{\ast} F=1$. If\, $F$ is a sheaf of abelian groups, then $R\nu_{\ast} F=F$.	
\end{Theorem}

The relevance to our study of $G$-torsors is then the following.

\begin{Proposition}\label{prop16}
	Let $G$ be any rigid group. Let $U\subseteq G$ be any rigid open subgroup, not necessarily normal. Then $G/U$ satisfies the approximation property.
\end{Proposition}

For example, both $\O^+/p^n$ and $\GL_n(\O^+/p^n)$ are of this form.
We use Proposition~\ref{prop16} to show that any $G$-torsor on $X$ admits a reduction of structure group to $U$ on an \'etale cover. 

\begin{Theorem}\label{p:reduction-of-structure-group-intro}
	Let $G$ be a rigid group over $K$, and let $U\subseteq G$ be a rigid open subgroup. Let $X$ be a sousperfectoid space over $K$, and let $\nu\colon X_v\to X_{\et}$ be the natural morphism of sites. Then the natural map
	\[ R^1\nu_{\ast}U\lra R^1\nu_{\ast}G\]
	is surjective. If\, $G$ is commutative, we more generally have 
	$R^k\nu_{\ast}U= R^k\nu_{\ast}G$
	for all $k\geq 1$.
\end{Theorem}

For non-commutative $G$, the map $R^1\nu_{\ast}U\to R^1\nu_{\ast}G$ is not in general an isomorphism.

The idea for the proof of  \cref{t:G-torsors-on-perfectoid-spaces} is now that by the theory of $p$-adic Lie groups, there is a large supply of open subgroups of $G$. If $G$ is commutative, then these are isomorphic to open subgroups of the Lie algebra of $G$ via the exponential, and one can deduce the result from the case of $\G_a$. In general, the exponential does not respect the group structure, but the relation to the Lie algebra suffices to trivialise $G$-torsors \'etale-locally by inductive lifting.

\medskip

Conceptually speaking, the main aim of this work is to launch a systematic study of \mbox{$G$-torsors} on adic spaces with a view towards generalisations and reformulations of the \mbox{$p$-adic} Simpson correspondence. Apart from the above results, we therefore prove some further foundational results  on $G$-torsors. We continue our study in \cite{heuer-sheafified-paCS}, where based on the results of this article, we give an explicit description of the sheaf $R^1\nu_{\ast}G$ on smooth rigid spaces, and use this to construct analytic moduli spaces of $G$-torsors on rigid spaces.

\section*{Setup and notation}
Throughout, let $p$ be a prime and $K$ a non-Archimedean field of residue characteristic $p$. We fix a subring $K^+\subseteq K$ of integral elements, which we often drop from notation. Let $\O_K$ be the ring of integers of $K$ and $\mathfrak m_K$ its maximal ideal. We also denote this by $\mathfrak m$ if $K$ is clear from context. For any $\epsilon>0$ for which there exists an element $a\in K$ with $|a|=|p|^\epsilon$, we denote by $p^\epsilon \mathfrak m_K$ the ideal of $\O_K$ which is the image of $\mathfrak m_K\subseteq \O_K$ under the multiplication by $a$ morphism $a\cdot \colon \O_K\to \O_K$.  This only depends on $\epsilon$: indeed, it is explicitly given by
$p^\epsilon\mathfrak m_K=\{x\in K \text{ s.t.\ }|x|<|p|^{\epsilon}\}$. We note that the right-hand side makes sense for any $\epsilon>0$, so by a mild abuse of notation, we can use this as a definition for $p^\epsilon\mathfrak m_K$ for any $\epsilon>0$.

 By an adic space over $K$, we mean an adic space over $\Spa(K,K^+)$ in the sense of Huber. All adic spaces occurring in this article will be adic spaces over $K$; in particular, they will automatically be Tate.
  For any adic space $X$, we use the \'etale site $X_{\et}$ in the sense of Kedlaya--Liu \cite[Definition 8.2.16]{KedlayaLiu-rel-p-p-adic-Hodge-I}.
By a rigid space over~$K$, we mean an adic space locally of topologically finite type over $\Spa(K,K^+)$. One example we use several times is the closed ball $\B^d=\Spa(K\langle T_1,\dots,T_d\rangle)$. We use perfectoid spaces in the sense of \cite{perfectoid-spaces}, as well as Scholze's category of diamonds; see  \cite{etale-cohomology-of-diamonds}.
Most adic spaces we consider throughout will be sousperfectoid in the sense of \cite[Section~6.3]{ScholzeBerkeleyLectureNotes} and~\cite{HK_sheafiness}: examples of such are smooth rigid spaces, perfectoid spaces, and products thereof. 

Throughout, we will want to work with adic spaces which are locally given by affinoid adic spaces $\Spa(A,A^+)$ such that for any finite \'etale morphism of Huber pairs $(A,A^+)\to (B,B^+)$, the pair $(B,B^+)$ is again sheafy. In order to express this, we follow Kedlaya--Liu and work with pre-adic spaces in the sense of \cite[Definition~8.2.3]{KedlayaLiu-rel-p-p-adic-Hodge-I}. Let us say that an adic space $X$ is \'etale-sheafy if for any pre-adic space $Y$ with an \'etale morphism of pre-adic spaces $Y\to X$, the space $Y$ is also adic.
For example, when $X$ is sousperfectoid, any pre-adic space $Y\to X$ is even a sousperfectoid adic space; hence sousperfectoid spaces are \'etale-sheafy. 

For any adic space $X$ over $K$, we consider the associated locally spatial diamond $X^\diamondsuit$ in the sense of \cite[Section~15]{etale-cohomology-of-diamonds}.
  As we have fixed a structure map to $K$, we may regard $X^\diamondsuit$ as a sheaf on the $v$-site $\Perf_{K,v}$ of perfectoid spaces over $K$. Then diamondification identifies the \'etale sites, see \cite[Lemma 15.6]{etale-cohomology-of-diamonds}, of $X$ and $X^\diamondsuit$. We may therefore often switch back and forth between $X$ and its associated diamond $X^\diamondsuit$. However,  there is one subtlety: on the diamantine side, we have a structure sheaf $\O$ on $X^\diamondsuit$ induced from that on $\Perf_{K,v}$.  We caution that we currently do not in general know whether the equivalence $X^\diamondsuit_\et=X_\et$ identifies the structure sheaves on both sides for sousperfectoid $X$, but this is true when $X$ is a perfectoid space or a semi-normal rigid space (see \cite[Theorem 8.7]{etale-cohomology-of-diamonds}, \cite[Proposition 10.2.3]{ScholzeBerkeleyLectureNotes}).

  We denote by $\Sous_{K}$  the category of sousperfectoid spaces over $K$.  We make this into a big \'etale site $\Sous_{K,\et}$ by equipping it with the \'etale topology. Second, we can also make sense of $v$-sheaves on $\Sous_{K}$: we define a morphism $X\to Y$ in $\Sous_{K}$ to be a $v$-cover if the associated morphism $X^\diamondsuit\to Y^\diamondsuit$ is. This does not make $\Sous_{K}$ into a site because the fibre product $X^\diamondsuit\times_{Y^\diamondsuit}X^\diamondsuit$ may not be represented by an object of $\Sous_{K}$. But this fibre product is still a diamond; hence it is covered by a perfectoid space in $\Sous_{K}$, which will suffice to define what it means to be a $v$-sheaf on $\Sous_{K}$ (see \Cref{d:approx-property}\eqref{d:ap-2}). We denote by $\mathrm{Dmd}_{K,v}$ the larger category of diamonds over $K$ equipped with the $v$-topology. As any diamond admits a $v$-cover by a perfectoid space, the categories of $v$-sheaves on $\Perf_{K,v}$, $\Sous_{K}$ and $\mathrm{Dmd}_{K,v}$ are all equivalent.

\subsection*{Acknowledgements}
We thank Johannes Ansch\"utz, Gabriel Dospinescu, Arthur-C\'esar Le Bras, Lucas Gerth, Ian Gleason, Lucas Mann, Peter Scholze, Annette Werner, Daxin Xu and Bogdan Zavyalov for very helpful conversations. We thank the referee for very helpful comments on an earlier version.
  
\section{Generalised representations on rigid spaces}\label{s:generalised-representations}

Before we discuss $G$-torsors for general rigid groups $G$, we study in this section the simpler case of \mbox{$v$-vector} bundles, that is, $\GL_n$-torsors on $X_v$. While this simplifies the setting, the overall line of argument is already the same as in the general case.
We begin by giving an alternative description of $v$-vector bundles that is of interest in the $p$-adic Simpson correspondence: the equivalence between $v$-vector bundles and Faltings' generalised representations.

\subsection{Generalised representations}

We start by adapting generalised representations as defined by Faltings, see \cite[Section~2 and Theo\-rem~3]{Faltings_SimpsonI}, from the algebraic setting of log schemes to an analytic setting: throughout this section, let $K$ be any non-Archimedean extension of $\Q_p$, and let $X$ be any \'etale-sheafy adic space over $K$. We recall that ``\'etale-sheafy'' was defined in the previous section and means that for any \'etale morphism of pre-adic spaces $Y\to X$ in the sense of \cite[Definition 8.2.16]{KedlayaLiu-rel-p-p-adic-Hodge-I}, the pre-adic space $Y$ is an adic space.

\begin{Definition}\leavevmode
	\begin{enumerate}
		\item  A \emph{generalised representation} on $X$ is a system $(M_k)_{k\in \N}$ of finite locally free $\O^+/p^k$ modules $M_k$ on $X_{\et}$ with isomorphisms $M_{k+1}/p^k\isomarrow M_k$ for all $k$. 
		
		\item The isogeny category is the localisation of the category of generalised representations on $X$ at multiplication by $p$.
		A \emph{generalised $\Q_p$-representation} on $X$ is a system of generalised representations $(M_{U_i})_{U_i\in \mathfrak U}$ on some \'etale cover $\mathfrak U$ of $X$ with isomorphisms $M_{U_i|U_i\times_X U_j}\to M_{U_j|U_j\times_X U_i}$ in the isogeny category that satisfy the cocycle condition.
	\end{enumerate}
\end{Definition}

\begin{Remark}
	The name ``generalised representation'' stems from the following observation: fix any base-point $x\in X(\overline{K})$, and consider  the \'etale fundamental group $\pi_1(X):=\pi^{\et}_1(X,x)$. Then to any semilinear continuous representation $\rho\colon \pi_1(X)\to \GL(E)$ on a finite free $\O_{K}$-module $E$, we can associate a generalised representation: namely, by continuity, the reduction $\rho_k\colon \pi_1(X)\to \GL(E/p^k)$ factors through a finite quotient of $\pi_1(X)$, corresponding to a finite \'etale cover $X'\to X$. We can then regard $\rho_k$ as a descent datum for a finite locally free $\O^+/p^k$-module $M_k$ along $X'\to X$, by defining $M_k$ on any $U\in X_{\et}$ to be
	\[ M_k(U):=\left\{x\in E(U\times_XX')/p^k\mid \gamma^\ast x= \rho^{-1}(\gamma)x \text{ for all }\gamma\in \pi_1(X)\right\}.\]
	The system $(M_k)_{k\in \N}$ then defines a generalised representation in the above sense.
\end{Remark}

The main goal of this short section is to prove the following equivalent characterisation.

\begin{Proposition}\label{t:generalised-reps-are-v-bundles}
  Let $X$ be any \'etale-sheafy adic space over a non-Archimedean field $K$ over $\Q_p$  $($e.g.\ $X$ can be any rigid spaces or sousperfectoid\,$)$.
	\begin{enumerate}
		\item\label{t:gravb-1} The morphism of sites $\nu\colon X_{v}\to X_{\et}$ induces a natural equivalence of categories
		\begin{align*}
			\{ \text{finite locally free $\O^+$-modules on $X_{v}$}\}&\lra \{ \text{generalised representations on $X$}\}\\
			V&\longmapsto \left(\nu_{\ast}\left(V/p^k\right)\right)_{k\in\N}.
		\end{align*}
		Moreover, these are equivalent to the category of finite locally free $\O^+$-modules on $X_{\qproet}$, and if\, $X$ is locally Noetherian also to that of finite locally free $\O^+$-modules on $X_{\proet}$.
		\item\label{i:generalised-Qp-reps} By localising at multiplication by $p$, this defines an equivalence of categories
		\begin{align*}
			\{ \text{finite locally free $\O$-modules on $X_{v}$}\}&\lra \{ \text{generalised $\Q_p$-representations on $X$}\}.
		\end{align*} 
		Moreover, these are equivalent to the category of finite locally free $\O$-modules on $X_{\qproet}$, and if\, $X$ is locally Noetherian also to that of finite locally free $\O$-modules on $X_{\proet}$.
	\end{enumerate}
\end{Proposition}

\begin{Definition}
	We also call a finite locally free $\O$-module on $X_{v}$  a \emph{$v$-vector bundle}. We then call finite locally free $\O$-modules on $X_{\et}$ ``\'etale vector bundles'' to clarify the topology.
\end{Definition}

From this perspective, one could reverse-engineer Faltings' definition of generalised $\Q_p$-representations by saying that they describe $v$-vector bundles purely in terms of $X_{\et}$.

\begin{Remark}
  The idea that generalised representations on a rigid space $X$ are vector bundles on $X_{\proet}$  is mentioned or implicit in other works:  in the arithmetic setting of smooth rigid spaces over discretely valued fields, it is hinted at by Liu--Zhu \cite{LiuZhu_RiemannHilbert} (see Remark~2.6). In more general analytic  settings, the idea appears in the works of W\"urthen \cite[Introduction, p.~4006]{wuerthen_vb_on_rigid_var} and Mann--Werner \cite{MannWerner_LocSys_p-adVB}, and a related result appears in the work of Morrow--Tsuji (\textit{cf.} \cite[Theorem 5.7]{MorrowTsuji}) in a good reduction setting. In more algebraic-geometric settings, related statements have been shown by Xu  \cite[Section~3.29]{XuTransport_parallele} and Yang--Zuo \cite[Lemma~2.10]{YangZuo-padicSimpson}.
	
	 \cref{t:generalised-reps-are-v-bundles} now provides a result in the full generality of adic spaces over $K$, free from any choice of integral model, or additional assumptions on $X$ or $K$. We note that this analytic setting introduces some additional subtleties due to the difference between $\O^+/p^k$-modules on $X_{\et}$ and modules on the reduction mod $p^k$ of some given formal model.
	
	The step to finite locally free $\O^+$-modules in the $v$-topology is newer, although we build on work of \cite{MannWerner_LocSys_p-adVB}. This step is relevant because $X_{\proet}$  is only defined for locally Noetherian $X$, while the $v$-topology is available more generally, which is useful even for rigid spaces, \textit{e.g.}~in the relative setting studied in \cite{heuer-sheafified-paCS}.
\end{Remark}

The proof will take up the entire section as we will also discuss some related topics to prepare the more general version for $G$-torsors.
The basic idea is to compare locally free $\O^+/p$-modules for various topologies. For this we can relax the setup of this section. 

\medskip

Let $K$ be any non-Archimedean field, not necessarily over $\Q_p$, and let $\varpi\in \O_K$ be any pseudo-uniformiser of $K$. The following proposition is the technical heart of this section. 

\begin{Proposition}\label{p:H^1-GL_n-mod-p-agrees-for-et-and-v}
	Let $X$ be an adic space over $K$ which is \'etale-sheafy.  Then for $m=0,1$ the map
	\begin{equation}\label{eq:H^mGL}
	H^m_{\et}(X,\GL_r(\O^+/\varpi))\longisomarrow H^m_{v}(X,\GL_r(\O^+/\varpi))
	\end{equation}
	is an isomorphism for all $r\in \N$. The analogous statement holds for $\GL_r(\O^+/\varpi\m)$.  
	
	In particular, the functor
	\[\nu^{\ast}\colon \Big\{ \begin{array}{c}\text{finite locally free}\\\text{$\O^+/\varpi$-modules on $X_\et$}\end{array}\Big\}\lra \Big\{ \begin{array}{c}\text{finite locally free}\\\text{$\O^+/\varpi$-modules on $X_v$}\end{array}\Big\}\]
	is an equivalence of categories, where $\nu\colon X_v\to X_{\et}$ is the natural morphism.
 \end{Proposition}

\begin{proof}
	To see the isomorphism  \eqref {eq:H^mGL} for $m=0$, it suffices to see that $\nu_{\ast}\GL_r(\O^+/\varpi)=\GL_r(\O^+/\varpi)$.
	
	Second, to see that the functor $\nu^\ast$ is fully faithful, it suffices to see that $\nu_{\ast}(\O^+/\varpi)=\O^+/\varpi$, \textit{i.e.} that~$H^0_{\et}(X,\O^+/\varpi)=H^0_{v}(X,\O^+/\varpi)$
	is an (honest, not just almost) isomorphism:
	indeed, for any two finite locally free $\O^+/\varpi$-modules $E_1$ and $E_2$ on $X_{\et}$, we have $\HOM(\nu^{\ast}E_1,\nu^{\ast}E_2)=\nu^{\ast}\HOM(E_1,E_2)$, and locally where the locally free $\O^+/\varpi$-module $\HOM(E_1,E_2)$ becomes trivial, the statement then follows.
	
	Third, to see the essential surjectivity, it suffices to see that $R^1\nu_{\ast}\GL_r(\O^+/\varpi)=1$. By passing to isomorphism classes, this will also imply the isomorphism \eqref{eq:H^mGL} for $m=1$.
	
	All of these three statements, and their analogues for $\O^+/\varpi\m$, will be proved by \Cref{p:approx-property-implies-et-v} below. Namely, they follow by approximation arguments, which we now axiomatise and discuss in detail in the following subsection.
	
	\subsection{Approximation of non-abelian sheaf cohomology}\label{s:approximation}
	For the proof of \cref{p:H^1-GL_n-mod-p-agrees-for-et-and-v}, we will use ideas from \cite[Section~3]{Scholze_p-adicHodgeForRigid}, \cite[Proposition~14.7]{etale-cohomology-of-diamonds}, \cite[Section~2]{MannWerner_LocSys_p-adVB}, \cite[Theorem 2.18]{heuer-Picard-good-reduction}. We  shall axiomatise the argument as we will apply it several times in the following section. Throughout this subsection, $K$ is any non-Archimedean field.
	We begin by recalling some properties of affinoid tilde-limits of adic spaces.
        
	\begin{Definition}[\textit{cf.} {\cite[Equation~(2.4.1)]{huber2013etale} and \cite[Section~2.4]{ScholzeWeinstein}}]\label{d:tilde-limits}
		Let  $(X_i)_{i\in I}$ be a cofiltered inverse system of affinoid adic spaces over $K$, and let $X$ be an affinoid adic space with compatible maps $X\to X_i$ for all $i\in I$. We write
		\[ X\sim \varprojlim_{i\in I} X_i\]
		if $|X|=\varprojlim |X_i|$ and if there is a cover of $X$ by affinoid opens $U$ for which the induced map $\varinjlim_{U_i}\O(U_i)\to \O(U)$
		has dense image, where the direct limit runs through all $i\in I$ and all affinoid opens  $U_i\subseteq X_i$ through which $U\to X_i$ factors.
	\end{Definition}

        \begin{Lemma}\label{l:tilde-lim-2-lim-on-etqcqs}
		In the situation of \cref{d:tilde-limits}, we have $X^\diamondsuit=\varprojlim_{i\in I} X_i^\diamondsuit$ for the associated diamonds. For the associated quasi-compact quasi-separated \'etale sites, we have 
		\[ X_{\etqcqs}=2\text{-}\varinjlim_{i\in I}X_{i,\etqcqs}.\]
	\end{Lemma}
        
	\begin{proof}
		The first part is \cite[Proposition~2.4.5]{ScholzeWeinstein}, the second  \cite[Proposition~11.23]{etale-cohomology-of-diamonds}.
	\end{proof}

	\begin{Definition}\label{d:tilde-limit-approx-version}
		In the situation of \cref{d:tilde-limits}, we write 
		\[ X\approx \varprojlim_{i\in I} X_i\]
		if already $\varinjlim_{i}\O(X_i)\to \O(X)$ has dense image.
	\end{Definition}

        One reason why this stronger notion is useful is the following elementary observation. 

        \begin{Lemma}\label{l:pointwise-approximation}
		Let $I$ and $J$ be directed sets, and let $(X_{i,j})_{i,j\in I\times J}$ be an inverse system of affinoid adic spaces over $K$, indexed over $I\times J$. Suppose that we have affinoid tilde-limits $X_i\approx \varprojlim_{j\in J} X_{i,j}$ for all $i$ as well as $X\approx \varprojlim_{i\in I}X_i$. Then $X\approx\varprojlim_{i,j\in I\times J} X_{i,j}$.
	\end{Lemma}
        
	\begin{proof}
		The condition on topological spaces is clear; the one on global sections follows by sequential approximation of elements.
	\end{proof}

	The crucial technical property of $\GL_n(\O^+/\varpi)$ that we use to prove \cref{p:H^1-GL_n-mod-p-agrees-for-et-and-v} is the following. 

	\begin{Lemma}\label{p:GL_n(O^+/p^n)-in-proet-site}
		Let $X\approx \varprojlim_{i\in I} X_i$ be an affinoid perfectoid tilde-limit of affinoid adic spaces over $K$. Then
		\[ H^0_{\et}(X,\GL_n(\O^+/\varpi))=\varinjlim_{i\in I} H^0_{\et}(X_i,\GL_n(\O^+/\varpi)).\]
		The analogous statement for $\GL_n(\O^+/\varpi\m)$ also holds.
	\end{Lemma}

        \begin{proof}
		Consider the Cartesian diagram of sheaves of sets
		\[\begin{tikzcd}
			\GL_n(\O^+/\varpi) \arrow[d,"\det"] \arrow[r] & M_n(\O^+/\varpi) \arrow[d,"\det"] \\
			(\O^+/\varpi)^\times \arrow[r] & \O^+/\varpi\rlap{.}
		\end{tikzcd}\]
		The statement holds for the bottom row by \cite[Lemma 3.10]{heuer-diamantine-Picard}. It thus also holds for $M_n(\O^+/\varpi)$, and thus for the top left since $H^0$ preserves finite limits.
		
		The statement for $\O^+/\varpi\m=\varinjlim_{\epsilon\to 0} \O^+/\varpi^{1+\epsilon}$ follows by taking colimits.
	\end{proof}
	
	For later applications, we more generally consider sheaves satisfying an analogous property in a more general setting. 
	
	\begin{Setup}\label{setup}
		Throughout the rest of this subsection, let $S$ be an adic space over $K$, and let $\mathcal C$ be a full subcategory of  adic spaces over $S$ satisfying the following conditions:
	\begin{enumerate}
		\item\label{ax1} $\mathcal C$ contains every perfectoid space $X$ over $S$ as well as $X\times \B^d$ for every $d\in \N$.
		\item\label{ax2} For any $X\in \mathcal C$ and any \'etale morphism of pre-adic spaces $Y\to X$ in the sense of \cite[Definition 8.2.16]{KedlayaLiu-rel-p-p-adic-Hodge-I}, the pre-adic space $Y$ is an adic space (\textit{i.e.}~$X$ is \'etale-sheafy) and $Y$ is contained in $\mathcal C$.
	\end{enumerate}
	We endow $\mathcal C$ with the \'etale topology.
	The requirement in \Cref{setup}.(2)
        ensures that this is well defined.
	\end{Setup}

        For example, $\mathcal C$ could be the big \'etale site $\Sous_{S,\et}$ of sousperfectoid spaces over $S$, which satisfies axioms~\eqref{ax1} and~\eqref{ax2} by \cite[Proposition 6.3.3]{ScholzeBerkeleyLectureNotes}. In particular, for $S=\Spa(K)$ it could be $\Sous_{K,\et}$.

        \begin{Definition}\label{d:approx-property}
		Let $\mathcal C$ be a site of adic spaces as in \Cref{setup}. Let $F$ be a sheaf on $\mathcal C$.
		\begin{enumerate}
		\item\label{d:ap-1}  We say that $F$ satisfies the \emph{approximation property} on $\mathcal C$ if for any affinoid perfectoid tilde-limit $X\approx \varprojlim_{i\in I} X_i$ of affinoid adic spaces in $\mathcal C$, we have \[F(X)=\varinjlim_{i\in I} F(X_i).\]
		\item\label{d:ap-2} Diamondification (see \cite[Section~15]{etale-cohomology-of-diamonds}) defines a natural morphism of sites $\lambda\colon \Dmd_{K,v}\to \mathcal C$. We call $F$ a \emph{$v$-sheaf on $\mathcal C$} if the natural map $F\to \lambda_{\ast}\lambda^\ast F$ is an isomorphism. Since any object of $\Dmd_{K,v}$ admits a cover by a perfectoid space over $S$, hence an object of $\mathcal C$, this means explicitly that for any $X,Y,Z\in \mathcal C$ and any $v$-covers $Y^\diamondsuit\to X^\diamondsuit$ and  $Z^\diamondsuit\to Y^\diamondsuit\times_{X^\diamondsuit}Y^\diamondsuit$,
		\[ F(X)\lra F(Y)\longrightrightarrows F(Z)\]
		is an equaliser diagram.
		\end{enumerate}
	\end{Definition}	
        
        The goal of this subsection is to show the following. 

        \begin{Proposition}\label{p:approx-property-implies-et-v}	Let $\mathcal C$ be a site of adic spaces as in \Cref{setup}.
		Let $F$ be a sheaf of sets on $\mathcal C$ satisfying the approximation property in the sense of \cref{d:approx-property}\,\eqref{d:ap-1}. Then:
		\begin{enumerate}
			\item\label{p:apiev-1} $F$ is already a $v$-sheaf.
			\item\label{p:apiev-2} If\, $F$ is a sheaf of groups, then for any $X$ in $\mathcal C$, we have
			$R^1\nu_{\ast} F=1$ for $\nu\colon X_v\to X_{\et}$. 
			\item\label{p:apiev-3} If\, $F$ is a sheaf of abelian groups, then $R^m\nu_{\ast} F=1$ for any $m\geq 1$.	
		\end{enumerate}
	\end{Proposition}

        The application we have in mind for the purpose of this section is to show the following. 

        \begin{Corollary}\label{cor215}
	  Let $X$ be any sousperfectoid space; then the map
          \[H^n_{\et}(X,\O^+/\varpi)\lra H^n_{v}(X,\O^+/\varpi)\]
		is an $($honest, not just almost\,$)$ isomorphism for all $n\geq 0$ .
	\end{Corollary}

        This is already non-trivial for $n=0$ and $X$ perfectoid, where almost acyclicity on affinoid perfectoid spaces (see \cite[Proposition 7.13]{perfectoid-spaces} for the \'etale topology, \cite[Proposition 8.8]{etale-cohomology-of-diamonds} for the $v$-topology) \textit{a priori} only implies an almost isomorphism. For rigid $X$, it is related to \cite[Lemma 4.2]{Scholze_p-adicHodgeForRigid}, which implies the analogous result for $X_{\proet}$. For strictly totally disconnected $X$, Corollary~\ref{cor215} recovers \cite[Proposition 2.13]{MannWerner_LocSys_p-adVB} by a similar proof.

	\begin{proof}[Proof of \cref{p:approx-property-implies-et-v}]
		We start with a few technical lemmas about the situation. 

	        \begin{Lemma}\label{l:approx-prop-and-cohom}
			Let $F$ be a sheaf of $($not necessarily abelian$)$ groups on $\mathcal C$ satisfying the approximation property. Then for $m=0,1$ and any $X\approx \varprojlim_{i\in I} X_i$ as in \cref{d:approx-property},
			\[H^m_{\et}(X,F)=\varinjlim_{i\in I} H^m_{\et}(X_i,F).\]
			If\, $F$ is abelian, this holds for any $m\geq 0$.
		\end{Lemma}

                \begin{proof}
			For $m=0$ this holds by definition. For $m\geq 1$ it follows by a \cH\ argument: following \cite[Definition 3.6]{heuer-diamantine-Picard}, let us call a morphism of affinoid adic spaces standard-\'etale if it is a successive composition of finite \'etale maps and rational open immersions. Such maps form a basis for the \'etale topology; hence every class in $H^m_{\et}(X,F)$ is trivialised by a standard-\'etale cover $X'$ of $X$. By \cref{l:tilde-lim-2-lim-on-etqcqs}, this cover arises via pullback from some $X_i'\to X_i$. For $j\geq i$, let $X_j':=X_j\times_{X_i}X_i'$; then by 
			\cite[Lemma~3.13]{heuer-diamantine-Picard}, we have a tilde-limit relation $X'\approx \varprojlim X_j'$, so this fulfils the assumptions of the lemma, and we therefore have $H^0_{\et}(X',F)=\varinjlim H^0_{\et}(X'_j,F)$. The statement for $H^m$ now follows by comparing \cH\ cohomology for the covers $X'\to X$ and $X'_j\to X_j$ for $j\geq i$.
		\end{proof}

                \begin{Lemma}\label{l:v-cover-approximate-by-balls}
			 Any morphism $f\colon Y\to X$ of affinoid perfectoid spaces over $K$ arises as the tilde-limit $Y\approx\varprojlim_{i\in I} Y_i$ of a cofiltered inverse system of sousperfectoid spaces $Y_i\subseteq \B^{d}\times X\to X$ that are rational open in unit balls over $X$. If moreover $X$ is strictly totally disconnected and $f$ is a $v$-cover, then each morphism $Y_i\to X$ admits a splitting.
		\end{Lemma}

                \begin{proof}
			This fact is used in the proof of \cite[Proposition 14.7]{etale-cohomology-of-diamonds}. For a proof, see \cite[Proposition~3.17]{heuer-diamantine-Picard}, and \cite[Lemma~2.23]{heuer-Picard-good-reduction} for the last sentence.
		\end{proof}

                \begin{Lemma}\label{l:proet-Colmez-cover}
		Let $X$ be any affinoid adic space over $K$. Then:
		\begin{enumerate}
			\item\label{l:pCc-1} There is an inverse system $(X_i\to X)_{i\in I}$ of finite \'etale Galois covers with affinoid perfectoid tilde-limit $X_\infty\approx \varprojlim X_i$. 
			\item\label{l:pCc-2} There is an inverse system $(X_i\to X)_{i\in I}$ of surjective \'etale morphisms with strictly totally disconnected tilde-limit $\wt X\approx \varprojlim X_i$.
		\end{enumerate}
		
		\end{Lemma}

                \begin{proof}
		Part~\eqref{l:pCc-1} is \cite[Lemma 10.1.6--7]{ScholzeWeinstein} (see also \cite[Lemma 15.3, Proposition 15.4]{etale-cohomology-of-diamonds}) and goes back to Colmez \cite[Section~4]{Colmez-espacesdeBanach}. 
		
		For part~\eqref{l:pCc-2}, we first recall that by \cite[Lemma 7.18]{etale-cohomology-of-diamonds}, the perfectoid space $X_\infty$ admits an affinoid pro-\'etale cover by a strictly totally disconnected space $\wt X$. By  \cite[Proposition~6.5]{etale-cohomology-of-diamonds}, this means that there is a cofiltered inverse system of  \'etale maps from affinoid perfectoid spaces $(X_{\infty,l}\to X_\infty)_{l\in L}$ such that $\wt X\approx \varprojlim_{l\in L} X_{\infty,l}$. Using Lemmas~\ref{l:tilde-lim-2-lim-on-etqcqs} and~\ref{l:pointwise-approximation}, these combine to give the desired system of \'etale morphisms over $X$. 
	        \end{proof}
                
	\begin{Lemma}\label{l:product-of-tilde-limits}
		Let $Y\approx \varprojlim_{i\in I} Y_i\to X$ be one of the affinoid perfectoid  tilde-limits of  Lemma~\ref{l:v-cover-approximate-by-balls} or~\ref{l:proet-Colmez-cover}. Then
		for any $m\in \N$, the $m$-fold fibre product $ Y_{/X}^{\times m}$ of\, $Y$ with itself over $X$ exists in the category of uniform adic spaces and is again affinoid perfectoid. Moreover, we still have
			\[ Y_{/X}^{\times m}\approx \varprojlim_{i\in I} Y_{i/X}^{\times m}.\]
	\end{Lemma}
        
	\begin{proof}
		In the case of \cref{l:v-cover-approximate-by-balls}, where $Y\to X$ is a morphism of affinoid perfectoid spaces, this can be seen exactly as in \cite[Lemma 2.8]{perfectoid-covers-Arizona}, by approximation of simple tensors.
		
		In the case of \cref{l:proet-Colmez-cover}, we can see inductively that the fibre product exists because for any $m\geq 2$, the projection $Y^{\times m}_{/X}\to Y^{\times (m-1)}_{/X}$ is pro-\'etale over an affinoid perfectoid space, hence itself affinoid perfectoid. 
		Moreover, for any $m\geq 2$, we have 
			\[Y^{\times m}_{/X}=Y\times_{X}Y^{\times (m-1)}_{/X}\approx \varprojlim_{i\in I}Y\times_{X}Y^{\times (m-1)}_{i/X}\]
		as this is true for affinoid pro-\'etale maps of perfectoid spaces. Second, for any fixed $i$, we have
			\[Y\times_{X}Y^{\times(m-1)}_{i/X}\approx\varprojlim_{j\in I} Y_j\times_{X}Y^{\times(m-1)}_{i/X}\]
			 by \cite[Lemma 3.13]{heuer-diamantine-Picard}. 
			 We can therefore invoke  \cref{l:pointwise-approximation} to deduce that
			 \[ Y_{/X}^{\times m}\approx \varprojlim_{i,j\in I\times J} Y_j\times_{X}Y^{\times(m-1)}_{i/X}. \]
			  Finally, the diagonal embedding $I\subseteq I\times I$  identifies $I$ with  a cofinal subset of $I\times I$; hence the cofiltered inverse systems $(Y_j\times_{X}Y^{\times(m-1)}_{i/X})_{i,j\in I\times J}$ and $(Y_i\times_{X}Y^{\times(m-1)}_{i/X})_{i\in I}$ are isomorphic.
	\end{proof}

	We can now give the proof of \cref{p:approx-property-implies-et-v}: for the reader interested in an exposition of the argument in the special case of $F=\O^+/p^n$, we also refer to \cite[Appendix C]{Zavyalov_acoh}.	
		
		For strictly totally disconnected $X$, part~\eqref{p:apiev-1} is proved in \cite[Lemma 2.12]{MannWerner_LocSys_p-adVB}. We repeat the argument here as it ties in well with the general case: let $\lambda\colon \Dmd_{K,v}\to \mathcal C$ be the natural morphism of sites, and let $F_v:=\lambda_{\ast}\lambda^{\ast}F$ be the $v$-sheafification of $F$ on $\mathcal C$ from \Cref{d:approx-property}.
		
\begin{enumerate}[label={\emph{Step}~\arabic*:}, ref=\arabic*,wide]
\item\label{step1} \emph{$F\to \lambda_{\ast}\lambda^{\ast}F$ is injective for strictly totally disconnected $X$.}~
		Let $\beta_1,\beta_2\in F(X)$ be such that the images of $\beta_1$ and $\beta_2$ under the map $F(X)\to F_v(X)$ agree. Then there is a $v$-cover $Y\to X$ by a perfectoid space such that the images of $\beta_1$ and $\beta_2$ under the map $F(X)\to F(Y)$ agree. Applying the approximation property to  the tilde-limit $Y\approx \varprojlim Y_i$ from \cref{l:v-cover-approximate-by-balls}, we see that $F(Y)=\varinjlim F(Y_i)$,
		so there is some cover $Y_i\to X$ such that the images of $\beta_1$ and $\beta_2$ already agree under  $F(Y)\to F(Y_i)$. But \cref{l:v-cover-approximate-by-balls} also says that the map $Y_i\to X$ has a section,  so it follows that $\beta_1=\beta_2$.

\item\label{step2} \emph{$F\to \lambda_{\ast}\lambda^{\ast}F$ is injective for any $X$ in $\mathcal C$.}~
		Let $\wt X\approx \varprojlim X_i$ be the strictly totally disconnected pro-\'etale cover of $X$ from \cref{l:proet-Colmez-cover}. Using that $F$ is an \'etale sheaf, then the approximation property, and finally Step~\ref{step1}, we find that the following map is injective:
		\[F(X)\longhookrightarrow \varinjlim F(X_i)=F(\wt X)\longhookrightarrow F_v(\wt X).\]
		As this factors through $F_v(X)$, this shows that $F(X)\to F_v(X)$ is injective. 
		
\item\label{step3} \emph{$F\to \lambda_{\ast}\lambda^{\ast}F$ is surjective for strictly totally disconnected $X$.}~
		Let $\alpha \in F_v(X)$; then there is an affinoid perfectoid $v$-cover $Y\to X$ such that $\alpha$ appears in $F(Y)$. By Step~\ref{step2} applied to $F(Y\times_XY)\hookrightarrow F_v(Y\times_XY)$, it thus lies in $\cH^0(Y\to X,F)$.
		
		 By the same approximation $Y\approx\varprojlim Y_i$ as in Step~\ref{step1}, as well as \cref{l:product-of-tilde-limits} and the approximation property, we have $\cH^0(Y\to X,F)=\varinjlim \cH^0(Y_i\to X,F)$. It follows that for $i$ large enough, $\alpha$ already lies in the image of 
		$\cH^0(Y_i\to X,F)\to F_v(X)$.
		
		We are not quite done yet because we do not require $Y_i\to X$ to be a cover in $\mathcal C$. To resolve this, observe that we can endow $\mathcal C$ with a ``smooth'' topology $\mathcal C_{\mathrm{sm}}$ for which we define covers to be the surjective morphisms $U\to V$ that are locally given by a composition $U\to V\times \mathbb B^d\to V$, where the first map is \'etale and the second map is the projection. This indeed defines a site, and this ``smooth'' topology clearly refines the \'etale one. We conclude that $\lambda\colon \Dmd_{K,v}\to \mathcal C$ factors through a morphism of sites $\mu\colon \mathcal C_{\mathrm{sm}}\to \mathcal C$. However, since $X$ is strictly totally disconnected, any cover of $X$ in $\mathcal C_{\mathrm{sm}}$ admits an \'etale refinement. For this reason, we see directly from the explicit definition of the sheafification in terms of the iteratively formed colimit of the $\supth{0}$ \v{C}ech cohomology over all covers  (see \textit{e.g.}~\cite[Tag~00W1]{StacksProject}) that $\mu^{\ast}\mathcal F(X)=\mathcal F(X)$. Since $\cH^0(Y_i\to X,F)\to F_v(X)$ factors through $\cH^0(Y_i\to X,\mu^\ast F)=\mu^{\ast}\mathcal F(X)$, this shows that $\alpha$ lies in the image of $F(X)\to F_v(X)$.
		
\item\label{step4} \emph{$F\to \lambda_{\ast}\lambda^{\ast}F$ is an isomorphism for any adic space $X$ in $\mathcal C$.}~
		As in Step~\ref{step2}, we consider the cover $\wt X\to X$ from \cref{l:proet-Colmez-cover}\eqref{l:pCc-2}. Let us compute $\cH^0(\wt X\to X,F)$. By Steps~\ref{step1} and~\ref{step3}, we have 
		\[ F\left(\wt X\right)=F_v\left(\wt X\right).\]
		Second, $\wt X\times_X\wt X$ is a perfectoid space by \Cref{l:product-of-tilde-limits}, hence in $\mathcal C$, so by Step~\ref{step2}, the map
		\[F\left(\wt X\times_X\wt X\right)\longhookrightarrow F_v\left(\wt X\times_X\wt X\right)\]
		is injective. Unravelling the definition of $\cH^0$, this combines to show that
		\[F_v(X)=\cH^0\left(\wt X\to X,F_v\right)=\cH^0\left(\wt X\to X,F\right).\]
		But we have $F(\wt X)=\varinjlim F(X_i)$, and by \cref{l:product-of-tilde-limits}, we moreover have $F(\wt X\times_X\wt X)=\varinjlim F(X_i\times_XX_i)$. This shows that
		\[
		\cH^0\left(\wt X\to X,F\right)=\varinjlim \cH^0(X_i\to X,F)=F(X).\]
		In combination, this shows that $F_v(X)=F(X)$, as we wanted to see.
		
		\medskip 
		
		We now prove the vanishing of $R^m\nu_{\ast}F$ for $m\geq 1$ by an intertwined induction on $m\geq 1$: the induction step is first carried out for strictly totally disconnected $X$, then for general $X$.

\item\label{step5} \emph{induction step to prove $R^m\nu_{\ast}F=1$  for strictly totally disconnected $X$.}~	
		As $X$ is strictly totally disconnected, it has trivial \'etale cohomology, so we need to show that $H^m_{v}(X,F)=0$.
		Let $\alpha \in H^m_{v}(X,F)$. By locality of cohomology, there is an affinoid perfectoid $v$-cover $Y\to X$ such that $\alpha$ becomes trivial in $H^m_{v}(Y,F)$.  
		
		For the non-abelian case and $m=1$, we now use the short exact sequence of pointed sets
		\[ \cH^1(Y\to X,F)\lra H^{1}_v(X,F)\lra H^1_v(Y,F).\]
		For the abelian case and  $m\geq 1$, we more generally have the \cH-to-sheaf spectral sequence 
		\[ \cH^k\left(Y\to X,H^j_v(-,F)\right)\Longrightarrow H^{k+j}_v(X,F)\]
		for $m=k+j$. In either case, we use the inverse system from \cref{l:v-cover-approximate-by-balls} and invoke Lemmas~\ref{l:product-of-tilde-limits} and~\ref{l:approx-prop-and-cohom}: by the induction hypothesis,
                we have for $j<m$ that
		\[ \cH^k\left(Y\to X,H^j_v(-,F)\right)=\cH^k\left(Y\to X,H^j_{\et}(-,F)\right)=\varinjlim\cH^k\left(Y_i\to X,H^j_{\et}(-,F)\right).\]
		As $Y_i\to X$ is split, the last term vanishes for $k>0$. Hence only the term for $k=0$, $j=m$ contributes. This means that $H^{m}_v(X,F)\to H^m_v(Y,F)$
		has trivial kernel. Hence $\alpha=0$.
		
\item\label{step6} \emph{induction step to prove $R^m\nu_{\ast}F=1$  for general $X$.}~
  Finally, we assume by the induction hypothesis
  that $R^j\nu_{\ast}F=1$ for any $1\leq j< m$ and any $X$, as well as for $j=m$ if $X$ is totally disconnected, and deduce that $R^m\nu_{\ast}F=1$. For this we again use the cover $\wt X\to X$ from \cref{l:proet-Colmez-cover}\eqref{l:pCc-2}. As in Step~\ref{step5}, we consider the short exact sequence of pointed sets, respectively the \cH-to-sheaf spectral sequence for abelian $F$
		\[ \cH^k\left(\wt X\to X,H^j_v(-,F)\right)\Longrightarrow H^{k+j}_v(X,F)\]
		for $k+j=m$.
		By the induction hypothesis, we have
		\[ \cH^k\left(\wt X\to X,H^j_v(-,F)\right)=\varinjlim_i  \cH^k\left(X_i\to X,H^j_{\et}(-,F)\right)\]
		for all $k>0$ and $j<m$, by Lemmas~\ref{l:product-of-tilde-limits} and~\ref{l:approx-prop-and-cohom}. For $j=m$, $k=0$, this equation also holds since by Step~~\ref{step5} we have $H^m_v(\wt X,F)=0$ and $\varinjlim_iH^m_{\et}(X_i,F)=H^m_{\et}(\wt X,F)=0$.  By comparing to the \cH-to-sheaf spectral sequence for the \'etale cover $X_i\to X$,
		\[ \cH^k\left(X_i\to X,H^j_{\et}(-,F)\right)\Longrightarrow H^{k+j}_{\et}(X,F),\]
		 we deduce that 
		 $H^{m}_{\et}(X,F)=H^{m}_v(X,F)$, as we wanted to see.\qedhere
                 \end{enumerate}
	\end{proof}
	
	This finishes the proof of \cref{p:H^1-GL_n-mod-p-agrees-for-et-and-v}.
\end{proof}

Before we continue, we record a lemma which we will use later. Recall that the $v$-sheaf property of \Cref{d:approx-property} means that we can extend $F$ uniquely to a sheaf on $\Dmd_{K,v}$.

\begin{Lemma}\label{l:approx-prop-for-fibre-products}
	Let $F$ be a sheaf of sets on $\mathcal C$ satisfying the approximation property in the sense of \cref{d:approx-property}. Let $(X_i)_{i\in I}$ be a cofiltered inverse system of locally spatial diamonds over $K$ with qcqs transition maps. Let $X:=\varprojlim_{i\in I} X_i$; this is a locally spatial diamond by \cite[Lemma~11.22]{etale-cohomology-of-diamonds}. Then
	\[ F(X)=\varinjlim_{i\in I} F(X_i).\]
\end{Lemma}

\begin{proof}
	We begin by following the proof of \cite[Lemma~11.22]{etale-cohomology-of-diamonds}, providing some additional details for the reader's convenience: since the statement is local on $X_0$ for any fixed $0\in I$, we can first reduce to the case that all $X_i$ and hence also $X$ are qcqs. By the ``refined argument'' in the proof of  \cite[Lemma~11.22]{etale-cohomology-of-diamonds}, we can then find a cofiltered inverse system of strictly totally disconnected spaces $(\wt X_i)_{i\in I}$ with compatible quasi-pro-\'etale surjections $\wt X_i\to X_i$: let us provide some details on this claim. Using \cite[Lemma~1.6]{Adamek-Rosicky} and arguing as in the proof of \cite[Corollary~1.7]{Adamek-Rosicky}, it suffices to consider the case that $I$ is a chain, \textit{i.e.}~that $I$ is the set of ordinals less than $\lambda$ for some ordinal $\lambda$. We can then argue by transfinite induction, as follows: given any limit ordinal $\mu<\lambda$, if we have found $\wt X_i\to X_i$ for all $i<\mu$, then 
	\[ \varprojlim_{i<\mu} \wt X_i\lra \varprojlim_{i<\mu} X_i\]
	is a quasi-pro-\'etale cover, hence so is its base-change
	\[ X'_\mu:=\varprojlim_{i<\mu} \wt X_i\times_{\varprojlim_{i<\mu} X_i}X_\mu \lra X_\mu.\]
	By the statement of \cite[Lemma~11.22]{etale-cohomology-of-diamonds}, $X'_\mu$ is a diamond, and we can define $\wt X_\mu\to X'_\mu$ to be any strictly totally disconnected quasi-pro-\'etale cover. Then $\wt X_\mu\to X_\mu$ has all desired properties.
	
	 By \cite[Proposition 6.5]{etale-cohomology-of-diamonds},  there exists an affinoid perfectoid tilde-limit
	\[ \wt X\approx \varprojlim_{i\in I} \wt X_i\]
	which is a quasi-pro-\'etale cover $\wt X\to X$. Since quasi-pro-\'etale covers of strictly totally disconnected spaces are again strictly totally disconnected by \cite[Lemma 7.19]{etale-cohomology-of-diamonds}, it follows that $\wt X_i\times_{X_i}\wt X_i\to X_i$ is still affinoid perfectoid. Since  limits in the category of diamonds commute with fibre products, we see  again by \cite[Proposition 6.5]{etale-cohomology-of-diamonds} that
	\[\wt X\times_X\wt X\approx \varprojlim_{i\in I} \wt X_i\times_{X_i}\wt X_i.\]
	We now use that $F$ is a $v$-sheaf. Combined with the approximation property, this implies
	\[ F(X)=\cH^0\left(\wt X\to X,F\right)=\varinjlim_{i\in I} \cH^0\left(\wt X_i\to X,F\right)=\varinjlim_{i\in I} F(X_i),\]
	as we wanted to see.
\end{proof}

\begin{Remark}
	For any sousperfectoid space $X$ over $K$ and any sheaf $F$ on $X_\et$, the restriction to $\Sous_{X,\et}$ of the pullback $\nu^{\ast}F$ along $\nu\colon X_v\to X_{\et}$ satisfies the approximation property by \cite[Proposition 14.9]{etale-cohomology-of-diamonds}. However, not every sheaf on $\Sous_{X,\et}$  satisfying the approximation property is of this form: for example, for $X=\Spa(K,\O_K)$ and $F=\O^+/p$ on $X_\et$, let $C$ be a complete algebraically closed field extension of~$K$ whose residue field is transcendental over that of $K$. Let $\overline{K}\subseteq C$ be the algebraic closure of $K$, and consider $Y:=\Spa(C)$ in $X_v$. Then one sees directly from the definition of $\nu^\ast$ that $\nu^\ast F(C)=\O_{\overline{K}}/p$, whereas $\O^+/p(C)=\O_C/p$.
\end{Remark}

\subsection{Generalised representations on perfectoid spaces}

With these preparations, we can now also prove an integral version of the following result of Kedlaya--Liu  already mentioned in the introduction (see also \cite[Lemma~17.1.8]{ScholzeBerkeleyLectureNotes}). 

\begin{Theorem}[\textit{cf.} {\cite[Theorem 3.5.8]{KedlayaLiu-II}}]\label{t:loc-free-O-modules}
	Let $X$ be a perfectoid space. Then the categories of vector bundles on $X_{\an}$, $X_{\et}$,  $X_{\proet}$,  $X_{\qproet}$ and $X_{v}$ are all equivalent.
\end{Theorem}

Our integral version is the following. 

\begin{Theorem}\label{t:loc-free-O^+-modules}
	Let $X$ be a perfectoid space. Then the categories of finite locally free $\O^+$-modules on $X_{\et}$,  $X_{\proet}$,  $X_{\qproet}$ and $X_{v}$ are all equivalent.
\end{Theorem}

\begin{Remark}
	We suspect that the categories might not be equivalent to the category of finite locally free $\O^+$-modules on $X_{\an}$, or even to that of finite projective $\O^+(X)$-modules if $X$ is affinoid perfectoid.
\end{Remark}

\begin{Remark}
 	It follows that finite locally free $\O^+$-modules on strictly totally disconnected spaces are trivial.
 	Already for $\Spa(\C_p)$, this becomes wrong in the almost category: for $c\in \mathbb R\backslash \Q$, the module $p^{c}\m$ becomes $\aeq \O^+$ on the $v$-cover $\Spa(C)\to \Spa(\C_p)$, where $C$ is any extension whose value group contains $c$.  We thank Lucas Mann for pointing this out to us.
\end{Remark}

We will later prove a much more general version in \cref{t:G-torsors-on-perfectoids}, the proposition being the case of $G=\GL_n(\O^+)$.  But for greater clarity in the much simpler case of $G=\GL_n(\O^+)$, we first give the results in this case and then later explain how to generalise.
 
 Recall that we use perfectoid spaces in the sense of \cite{perfectoid-spaces}, so $X$ is assumed to live over a perfectoid field $(K,K^+)$, and we denote by $\varpi$ any pseudo-uniformiser of $K$.
 
 We now first observe the following.

\begin{Lemma}\label{l:local-freeness--in-the-limit-in-the-v-topology}
	Let $X$ be an affinoid perfectoid space. 
	\begin{enumerate}
		\item\label{l:226-1} Let $V$ be a $\varpi$-torsion-free $\O^+$-module on $X_v$ such that $V=\varprojlim_k V/\varpi^k$. Assume that there is an $r\in \N$ such that $V/\varpi\m\cong \O^{+r}/\varpi\m$ as $\O^+$-modules. Then $V\cong \O^{+r}$. 
		\item\label{l:226-2} Let $V$ be a finite locally free $\O^+/\varpi$-module. If\, $V/\m$ is free over $\O^+/\m$, then $V$ is free.
	\end{enumerate}
\end{Lemma}

\begin{proof}
	Since $V$ is $\varpi$-torsion-free, we have for any $k>0$ a short exact sequence on $X_v$
	\[ 0\lra \m V/\varpi\mathfrak m\xrightarrow{\varpi^k} V/\varpi^{k+1}\mathfrak m\lra V/\varpi^k\m\lra 0.\]
	For part~\eqref{l:226-1}, assume that $V/\varpi^k\m$ is free of rank $r$, and let $v_1,\dots,v_r$ be any basis of $V/\varpi^k\m(X)$ as a finite free $\O^{+}/\varpi^k\m(X)$-module. Consider the long exact sequence
	\[ 0\lra \m V/\varpi \m(X)\lra V/\varpi^{k+1}\m(X)\lra V/\varpi^k\m(X)\lra H^1_v(X,\m V/\varpi\m).\]
	The last term vanishes: indeed, we have $\m V/\varpi\m\cong \m \O^{+r}/\varpi\m= \m \otimes_{K^+}\O^{+r}/\varpi\m$ by assumption and moreover $H^1_v(X,\O^+/\varpi\m)\aeq 0$ by almost acyclicity; hence
	\[ H^1_v(X,\m V/\varpi\m)\cong H^1_v(X,\O^{+r}/\varpi\m)\otimes_{K^+} \m=0\]
	vanishes ``without almost'', because any almost zero module vanishes after tensoring with $\mathfrak m$.
	This shows that we can lift the basis vectors to sections  $v_1',\dots,v_r'\in V/\varpi^{k+1}\m(X)$. Starting with $k=1$, we can thus inductively define  a map
	\[ \phi_{k+1}\colon \O^{+r}/\varpi^{k+1}\m\lra V/\varpi^{k+1}\m\]
	which reduces mod $\varpi^k\m$ to $\phi_k$. By the 5-lemma, $\phi_{k+1}$ is an isomorphism. In the limit, due to the completeness assumption on $V$, we obtain an isomorphism
	$\phi=\varprojlim_k\phi_k\colon \O^{+r}\to V$.
	
	To deduce part~\eqref{l:226-2}, observe that we can by the same argument lift any basis of $V/\m$ to sections of $V/\varpi \m$: indeed, the lifting obstruction again lies in $H^1_v(X,\m V/\varpi\m)=0$. This basis defines a map $\phi\colon \O^{+r}/\varpi\m\to V/\varpi\m$ that is an isomorphism mod $\m$. If we know \textit{a priori} that $V/\varpi\m$ is finite locally free, we can consider $\det \phi\in \wedge^r(V/\varpi\m)(X)$, which is an invertible  section mod $\m$, thus invertible. This shows that $\phi$ is an isomorphism.
\end{proof}

\begin{proof}[Proof of \cref{t:loc-free-O^+-modules}]
	We clearly have a chain of fully faithful functors, so it suffices to prove that any finite locally free $\O^+$-module $V$ on $X_{v}$ is already free \'etale-locally. For this we may assume that $X$ is affinoid perfectoid. Then by \cref{p:H^1-GL_n-mod-p-agrees-for-et-and-v}, there is an affinoid perfectoid \'etale cover $X'\to X$ on which $V/\varpi\m$ becomes trivial. By \cref{l:local-freeness--in-the-limit-in-the-v-topology}, already $V$ is trivial on $X'$.
\end{proof}

\begin{Corollary}\label{c:v-vb-on-Spa(K)}
	Let $X=\Spa(K,K^+)$, where $K$ is a perfectoid field. Then any locally free $\O^+$-module on $X_v$ is trivial.
\end{Corollary}

\begin{proof} 
	Any \'etale cover of $X$ is a disjoint union of maps $\Spa(L,L^+)\to \Spa(K,K^+)$, where $L|K$ is finite Galois and $L^+|K^+$ is the integral closure, hence faithfully flat. Any $v$-vector bundle on $X$ thus defines a finite projective $K^+$-module. This is free as $K^+$ is a valuation ring.
\end{proof}

As any adic space over $K$ has a pro-\'etale perfectoid cover by \cref{l:proet-Colmez-cover}, we deduce the following. 

\begin{Corollary}\label{c:vector-bundles-on-diamonds-different-tops}
	Let $X$ be any diamond; then the categories of finite locally free $\O^+$-modules on $X_{\qproet}$ and  $X_{v}$ $($and $X_{\proet}$ if\, $X$ is locally Noetherian and $\operatorname{char} K=0)$ are equivalent. 
\end{Corollary}

We now return to the proof of \cref{t:generalised-reps-are-v-bundles}. For the second part, we also need the following, which we will also generalise later in \cref{p:reduction-of-structure-group-to-open-subgroup}, by a different proof. 

\begin{Lemma}\label{l:v-bundle-is-et-locally-small}
	Let $X$ be an adic space over $K$. Let $V$ be a $v$-vector bundle on $X$. Then \'etale-locally on $X$, there is a finite $v$-locally  free $\O^+$-module $V^+$ such that $V^+\tf=V$.
\end{Lemma}

\begin{proof}
	We may assume that $X$ is affinoid.
	By \cref{l:proet-Colmez-cover}, there is then a pro-finite-\'etale affinoid perfectoid cover $X_\infty\to X$ that is Galois for some profinite group $N$. The pullback of $V$ to $X_\infty$ is \'etale-locally free by \cref{t:loc-free-O-modules}. By \cref{l:tilde-lim-2-lim-on-etqcqs}, we can replace $X$ by some \'etale cover to assume that the pullback of $V$ to $X_\infty$ is trivial.
	
	It follows that $V$ is associated to a descent datum on the trivial vector bundle on $X_\infty$, thus defines a class in the first term of the Cartan--Leray sequence (see \cite[Proposition~2.8]{heuer-v_lb_rigid})
	\[ 0\lra H^1_{\cts}(N,\GL_n(\O(X_\infty)))\lra H^1_v(X,\GL_n(\O))\lra  H^1_v(X_\infty,\GL_n(\O)).\]
	Let $\rho\colon N\to \GL_n(\O(X_\infty))$ be any continuous $1$-cocycle representing $V$. Let $N_0$ be the inverse image of the open subspace $\GL_n(\O^+(X_\infty))\subseteq \GL_n(\O(X_\infty))$. Then by the continuity of $\rho$, the subspace $N_0\subseteq N$ is an open neighbourhood of the identity in $N$, so $N_0$ contains  an open subgroup $N_1\subseteq N$. Since $N/N_1$ is finite, this corresponds to a finite \'etale cover $f\colon X'\to X$, and the pullback of $V$ to $X'$ is defined by the $1$-cocycle $\rho\colon N_1\to \GL_n(\O^+(X_\infty))$. By the functoriality of the Cartan--Leray sequence, this defines an element in $H^1_v(X',\GL_n(\O^+(X_\infty)))$ whose image in $H^1_v(X',\GL_n(\O))$ corresponds to the isomorphism class of $f^{\ast}V$.\looseness=1
\end{proof}

We can now prove the equivalence of $v$-vector bundles and generalised representations. 

\begin{proof}[Proof of \cref{t:generalised-reps-are-v-bundles}]
	The equivalence between the categories of finite locally free modules in various topologies is \cref{c:vector-bundles-on-diamonds-different-tops}.
	We now first consider the functor in part~\eqref{t:gravb-1}:
	 the $\O^+/p^n$-module $V/p^n$ is already locally free in the \'etale topology by \cref{p:H^1-GL_n-mod-p-agrees-for-et-and-v}, so  $\nu_{\ast}(V/p^n)$ is a locally free $\O^+/p^n$-module on $X_{\et}$. Thus the functor is well defined.
	
	It is fully faithful by \cref{p:H^1-GL_n-mod-p-agrees-for-et-and-v} and because any system of compatible morphisms $V/p^n\to W/p^n$ on~$X_v$ in the limit induces a unique $\O^+$-linear morphism $V\to W$. 
	
	It remains to see that the functor is essentially surjective. Let $(M_n)_{n\in\N}$ be a generalised representation, and consider the $v$-sheaf $V:=\varprojlim_{n\in\N} \nu^{\ast}M_n$. As the $v$-topology is replete in the sense of \cite[Section~3.1]{bhatt-scholze-proetale} (see \cite[Lemma~2.6]{heuer-Picard-good-reduction}), we have $V/p^n=M_n$. To see that $V$ is finite locally free, let $X'\to X$ be an \'etale cover that trivialises $M_1$. Choose a pro-\'etale cover by an affinoid perfectoid $X''\to X'$. Then by \cref{l:local-freeness--in-the-limit-in-the-v-topology}, the fact that $V$ is $p$-adically complete and $V/p$ is free on $X''$ implies that $V$ is a finite free $\O^+$-module on~$X''$. This shows that $V$ is a finite locally free $\O^+$-module on $X_v$.
		
	\medskip
	
	It thus remains to prove \cref{t:generalised-reps-are-v-bundles}\eqref{i:generalised-Qp-reps} about generalised $\Q_p$-representations.
	It is clear from Theorem~\ref{t:loc-free-O^+-modules} that the data in the definition of generalised $\Q_p$-representations translates into gluing data for vector bundles on $X_{\proet}$. We thus have a fully faithful functor 
	\[\left\{ \text{generalised $\Q_p$-representations on $X_{\et}$}\right\}\lra \left\{ \text{finite locally free $\O$-modules on $X_{\proet}$}\right\}.\]
	
	This is essentially surjective if any finite locally free $\O$-modules on $X_{\proet}$ comes from a finite locally free $\O^+$-module on an \'etale cover, which is guaranteed by \cref{l:v-bundle-is-et-locally-small}.
\end{proof}

\section{\texorpdfstring{$\boldsymbol{G}$}{G}-torsors for rigid groups \texorpdfstring{$\boldsymbol{G}$}{G} on adic spaces}
We now pass from vector bundles to torsors under any rigid analytic group $G$, the $p$-adic analogue of a complex Lie group. We start by discussing some background on rigid analytic group varieties, since non-commutative rigid groups are not so common in the literature. For this reason, and to provide a reference for future articles, we discuss slightly more than is strictly necessary to prove the main result of this article.
\medskip

From now on, we assume that $K$ is a perfectoid field over $\Q_p.$ The characteristic $0$ assumption is necessary in this context to obtain a $p$-adic exponential.

\begin{Definition}
  By a rigid analytic group variety, or just \emph{rigid group}, we mean a group object $G$ in the category of adic spaces locally of topologically finite type over $\Spa(K,K^+)$.
 \end{Definition}

We refer to \cite[Sections~1.2--1.3]{Fargues-groupes-analytiques} for some background on rigid groups, some of which we recall below.
  Rigid groups have been studied primarily in the commutative case, \textit{e.g.}~in \cite{Lutkebohmert_structure_of_bounded}, but we do not assume $G$ to be commutative. We therefore write the group operation multiplicatively as $m\colon G\times G\to G$, and write $1\in G$ for the identity.
As before, we freely identify $G$ with its associated $v$-sheaf over $K$. This is harmless in characteristic $0$ due to the following. 

\begin{Lemma}[\textit{cf.} {\cite[Proposition 1]{Fargues-groupes-analytiques}}]\label{l:G-smooth}
	Any rigid group $G$ is smooth. Moreover, there is a rigid open subspace $1\in U\subseteq G$ for which there is an isomorphism $U\isomarrow \B^d$ of rigid spaces.
\end{Lemma}

\begin{Example}\label{ex:rigid-groups}\leavevmode
	\begin{enumerate}
\item 
For any algebraic group $G$ over $\Spec(K)$, we get an associated rigid analytic group by analytification. Again, we often identify $G$ with its analytification as well as with the associated $v$-sheaf. If $G$ is affine, then the latter can explicitly be described in terms of the algebraic group as being the sheaf $G(\O)$ on $\Perf_K$ sending $(R,R^+)$ to $G(R)$. We are particularly interested in the case $G=\GL_n$. 
If $G$ is not affine, the description of the $v$-sheaf is still true after sheafification. 

Two examples that we use frequently throughout are $G=\G_a$, which is the rigid affine line with its additive structure and represents $\O$, as well as $G=\G_m$, which is the rigid affine line punctured at the origin with its multiplicative structure and represents $\O^\times$.

More generally, for any finite-dimensional vector space $W$ over $K$, we have the rigid group $W\otimes_K\G_a$. We call any rigid group of this form a \emph{rigid vector group.}
\item Let $\mG$ be a smooth formal group scheme over $\Spf(K^+)$, \textit{i.e.}~a group object in the category of formal schemes locally of topologically finite presentation over $\Spf(K^+)$ that is smooth over $\Spf(K^+)$. Then the adic generic fibre $\mathcal G^{\ad}_{\eta}\to\Spa(K,K^+)$ in the sense of \cite[Section~2.2]{ScholzeWeinstein} is naturally a rigid group. 
We say that a rigid group $G$ over $K$ has \emph{good reduction} if it arises in this way, \textit{i.e.}~if there is a smooth formal group scheme $\mathcal G$ over $\Spf(K^+)$ whose adic generic fibre  is isomorphic as a rigid group to $G$.
\item Assume that $G$ is an algebraic group over $K$ that extends to a smooth algebraic group $G_{K^+}$ over~$K^+$.  For instance, by a theorem of Chevalley--Demazure, such a model always exists if $G$ is a split connected reductive group, namely the Chevalley model; see \cite[Corollaire XXV.1.2]{SGA3}. Consider the $p$-adic completion  $\mathcal G$ of $G_{K^+}$. Then the adic generic fibre of $\mathcal G$ is an open rigid subgroup $G^+\subseteq G$ that has good reduction. 

For $G=\G_a$ with its canonical extension to $K^+$, this construction yields the closed unit ball $\G_a^+\subseteq \G_a$ which represents the $v$-sheaf $\O^+$.
For $G=\GL_n$, it recovers the rigid open subgroup $\GL_n(\O^+)\subseteq \GL_n(\O)$ of integral matrices from the last section.
\end{enumerate}
\end{Example}

\subsection{The correspondence between rigid groups and Lie algebras}

For any rigid group $G$ over $(K,K^+)$, we denote by 
\[\mathfrak g:=\operatorname{Lie} G:=\ker\left(G(K[X]/X^2)\to G(K)\right)\]
the \emph{Lie algebra} of $G$, defined exactly as for algebraic groups (see also \cite[Section~1.2]{Fargues-groupes-analytiques}). Its underlying $K$-vector space is the tangent space of $G$ at the identity, so $\dim \mathfrak g=\dim G$. As for algebraic groups, we can  regard $\mathfrak g$ as a rigid vector group over $K$, explicitly given by $\operatorname{Lie} G\otimes_K \G_a$. As usual, we also consider this as a $v$-sheaf on $\Perf_K$, explicitly given for any perfectoid $K$-algebra $(R,R^+)$ by $\mathfrak g(R)=\ker(G(R[X]/X^2)\to G(R))$.
We shall therefore from now on write $\mathfrak g(K)$ when we mean the underlying $K$-vector space. 

One application of the $v$-sheaf perspective is that it immediately shows that we have a rigid analytic \emph{adjoint action}
\[ \ad\colon G\lra \GL(\mathfrak g),\]
a homomorphism of rigid groups sending $g\in G(R)$ to the $R$-linear automorphism of $\mathfrak g(R)$ induced on tangent spaces by the conjugation map $G\times_KR\to G\times_KR$ that sends $h$ to $g^{-1}hg$. On tangent spaces, this induces a map $\ad\colon \mg\to \End(\mathfrak g)$ which defines the Lie bracket on $\mg$.

\medskip

With these definitions, there is a $p$-adic analogue of the Lie group-Lie algebra correspondence in complex geometry,  the main difference to the complex case being that one needs to account for the fact that in the $p$-adic setting there are many more open subgroups.  While we do not know of a place in the literature where this is discussed in the present setting of rigid analytic groups, the construction is of course essentially classical as it follows immediately from the theory of $p$-adic Lie groups, as developed in \cite{Bourbaki-Lie,Serre_Lie,Schneider_Lie}. 

\begin{Theorem}[\textit{cf.} {\cite[Part~II]{Serre_Lie}}]\label{t:Lie-grp-Lie-alg-correspondence}
	Sending $G$ to $\mathrm{Lie}(G)$ defines a functor
	\[ \mathrm{Lie}\colon \{  \text{rigid groups over $K$}\}\lra \{\,\text{fin.-dim.\ Lie algebras over $K$}\}.\]
	This
	 becomes an equivalence of categories after localising the left-hand side at the class of homomorphisms of rigid groups that are open immersions.
	
	Moreover, for any homomorphism  $f\colon G\to H$ of rigid groups for which the morphism $\mathrm{Lie}(f)\colon \mathrm{Lie}(G)\to \mathrm{Lie}(H)$ is an isomorphism, there are rigid open subgroups $G_0\subseteq G$ and $H_0\subseteq H$ such that $f$ restricts to an isomorphism of rigid groups $G_0\to H_0$.
\end{Theorem}

\begin{proof}
	Due to \cref{l:G-smooth}, there are open neighbourhoods of $1\in G$ and $1\in H$ inside of which rigid open subdiscs around $1$ with morphisms of rigid spaces between them are equivalent to open subdiscs of $G(K)$ and $H(K)$ centred at $1$ and analytic maps between them in the sense of $p$-adic Lie groups as defined in \cite[Part~II, Chapter~II]{Serre_Lie}. In particular, after localisation the category on the left becomes  equivalent to Serre's ``analytic group chunks''.
	
	Again by \cref{l:G-smooth}, the completion $\wh G$ of $G$ at $1$ is isomorphic as a formal scheme to $\Spf(K[[X_1,\dots,X_d]])$ and inherits the structure of a formal group law. This defines a functor
	\[R\colon\{\text{$d$-dim.\ rigid groups over $K$}\}\lra \{\text{$d$-dim.\ formal group laws over $K$}\}.\]
	On the other hand, sending a formal group law $H$ over $K$ to the tangent space $\mathfrak h$ at the origin defines an equivalence of categories, see \cite[Part II, Chapter V.6]{Serre_Lie}, 
	\[S\colon\{\text{$d$-dim.\ formal group laws over $K$}\}\lra \{\text{$d$-dim.\ Lie algebras over $K$}\}.\] 
	The composition is easily seen to coincide with $\mathrm{Lie}$. 
	
	It thus suffices to prove the results for the functor $R$: that $R$ is essentially surjective follows from \cite[Part II, Chapter IV.8, Theorem~1]{Serre_Lie} (or \cite[Proposition 17.6]{Schneider_Lie}). It becomes full after the localisation by \cite[Part II, Chapter V.7, Theorem~1]{Serre_Lie}. It becomes faithful after the localisation because any morphism from a connected rigid group is determined on any open subgroup by Zariski density.
	
	The last sentence follows from \cite[Part II, Chapter V.7, Corollary~1.2]{Serre_Lie}.   For an alternative proof of this statement, see also \cite[Section~1, Lemmes 1 et~2]{Fargues-groupes-analytiques}.
\end{proof}

\subsection{The $\boldsymbol{p}$-adic exponential of a rigid group $\boldsymbol{G}$}\label{s:exp-log}
For the rest of this section, we fix any rigid group $G$ over $(K,K^+)$. 
We now give a brief account of the $p$-adic exponential in the non-commutative setting of rigid groups:  exactly like for \cref{t:Lie-grp-Lie-alg-correspondence}, this  is essentially a translation of a classical result from the theory of $p$-adic Lie groups into the setting of rigid groups. \looseness=-1

\begin{Proposition}\label{p:general-exponential}
	Let $G$ be a rigid group over $K$ and $\mathfrak g$ its Lie algebra. Then there is an open $\O_K$-linear subgroup $\mathfrak g^{\circ}\subseteq \mathfrak g$ for which there is a unique open immersion of rigid spaces
	\[ \exp\colon\mathfrak g^{\circ}\lra G\]
	with $\exp(0)=1$ that induces the identity $\mathfrak g\to \mathfrak g$ on tangent spaces and makes the diagram
	\[\begin{tikzcd}
		\mathfrak g^{\circ}\times \mathfrak g^{\circ} \arrow[d, "\mathrm{BCH}"'] \arrow[r, "\exp"] & G\times G \arrow[d, "m"] \\
		\mathfrak g^{\circ} \arrow[r, "\exp"] & G
	\end{tikzcd}\]
	commute,
	where $\mathrm{BCH}$ is the Baker--Campbell--Hausdorff formula \[\mathrm{BCH}(x,y)=x+y+\tfrac{1}{2}[x,y]+\tfrac{1}{12}[x-y,[x,y]]+\cdots.\]
        It has the following properties:
	\begin{enumerate}
		\item\label{p:ge-1} For any open subgroup $\mathfrak g_1\subseteq \mathfrak g^{\circ}$, the map $\exp\colon \mathfrak g_1\to G$ is an isomorphism of rigid spaces onto an open subgroup of\, $G$ $($but not necessarily a homomorphism$)$.
		\item\label{p:ge-2} 
	 $\exp$ is functorial in $G$ $($i.e.\ after shrinking $\mg^{\circ}$, the obvious diagram commutes$)$.
	\end{enumerate}
\end{Proposition}

\begin{Definition}
	In particular, part~\eqref{p:ge-1} says that $\exp\colon \mg^\circ\to G$ is an isomorphism onto an open subgroup $G^\circ\subseteq G$. We denote the inverse by $\log\colon G^\circ\to \mg^\circ$.
\end{Definition}

\begin{Remark}\label{ex:exp}\leavevmode
	\begin{enumerate}
		\item 
		If $G$ is commutative, then $\mathrm{BCH}(x,y)=x+y$ and the diagram says that $\exp$ is an isomorphism of rigid groups. This case is discussed in \cite[Section~1.5]{Fargues-groupes-analytiques}.
	      \item There is in general no canonical way to choose the subgroup $\mathfrak g^\circ$: already if $G=\GL(W)$ for some finite-dimensional $K$-vector space $W$, we need a basis of $W$ to get $\mathfrak g^\circ$. In the following, we shall therefore fix a choice of $\mathfrak g^\circ$ and thus $G^\circ$, but this choice will be harmless as we will always be free to replace $\mathfrak g^\circ$ and $G^\circ$ by open subgroups. More canonically, one could consider the filtered  system of all open subgroups on which $\exp$ is defined.
		\item For $G=\GL_n$, the exponential can be explicitly described by the usual formulas: we have $\mg=M_n$, and for any uniform Huber pair  $(R,R^+)$ over $(K,K^+)$, the $p$-adic exponential and logarithm series define continuous homomorphisms
		\begin{alignat*}{3}
			\exp&\colon&\mg^\circ:=p^{\alpha_0}\m M_n(R^+)&\lra 1+p^{\alpha_0}\m M_n(R^+),\quad & x&\longmapsto \textstyle\sum_{n=0}^\infty \frac{x^n}{n!},\\
			\log&\colon&1+\mathfrak mM_n(R^+)&\lra M_n(R),\quad & 1+x&\longmapsto -\textstyle\sum_{n=1}^\infty \frac{(-x)^n}{n}
		\end{alignat*}
		which are mutually inverse when restricted to the domain and codomain of the exponential, where we set $\alpha_0:=1/(p-1)$ if $p>2$ and $\alpha_0=1/4$ otherwise. More generally, this describes $\exp$ for linear algebraic groups, or more generally any rigid group $G$ that admits an injective homomorphism $G\hookrightarrow \GL_n$: the maps $\exp$ and $\log$ restrict to the closed subgroups $\mg^\circ:=\mg\cap p^{\alpha_0}\m M_n(\O^+)$ and $G^\circ:=G\cap (1+\mathfrak p^{\alpha_0}\m M_n(\O^+))$.
	\end{enumerate}
\end{Remark}

\begin{proof}
	This is developed in \cite[Chapitre~II]{Bourbaki-Lie} and \cite[Corollary~III.18.19]{Schneider_Lie} for $p$-adic Lie groups instead of rigid groups, but since $p$-adic Lie groups are locally analytic, one can obtain from this a morphism of rigid groups on a neighbourhood of the identity.	
	We sketch the proof.  
	
	We use the functor $S$ already mentioned in the proof of \cref{t:Lie-grp-Lie-alg-correspondence}.
	Its inverse is given by sending a Lie algebra $\mathfrak h$ to the formal group law $F_{\mathfrak h}:=U(\mathfrak h)^{\ast}$, where $U(\mathfrak h)^{\ast}$ is the $K$-linear dual of the universal enveloping algebra with its natural co-algebra structure.
	
	Applying this construction to $\mathfrak g$,  Poincar\'e--Birkhoff--Witt's theorem  identifies $U(\mathfrak g)^{\ast}$ with the completed symmetric algebra $K[[\mathfrak g^{\ast}]]$ of $\mg^\ast$; see \cite[Proposition~18.3]{Schneider_Lie}. This defines an isomorphism of formal schemes $\psi\colon F_{\mg}\to \widehat{\G}_a\otimes \mg$. Then $\mathrm{BCH}$ expresses the formal group law on $\widehat{\G}_a\otimes \mg$ transported from  $F_{\mg}$ via $\psi$; see \cite[Corollary~18.15]{Schneider_Lie}.
	
	The crucial calculation is now that $\mathrm{BCH}$ converges on an open ball in $\mg$: namely, choose any $K$-basis of $\mathfrak g$, and for $k\geq 0$ let $\mg^\circ_{k}$ be the $p^{k}\m\G_a^+$-span of the basis vectors, considered as an open ball. Then by \cite[Proposition~17.6]{Schneider_Lie}, $\mathrm{BCH}$ defines for $k\gg 0$ a function
	\[\mathrm{BCH}\colon\mg^\circ_{k}\times \mg^\circ_{k}\lra \mg^\circ_{k}\]
	that endows $\mg^\circ_{k}$ with the structure of a rigid group  with Lie algebra $\mathfrak g$. Let us call this $\mg^\circ_{k,\BCH}$.
	
	By the above equivalence of categories $S$, there is now a natural isomorphism between the formal groups associated to  $\mg^\circ_{k,\BCH}$ and $G$ that is the identity on tangent spaces. By \cref{t:Lie-grp-Lie-alg-correspondence} and its proof, this is analytic in a neighbourhood of the identity, \textit{i.e.}~for $k\gg 0$ there is an open subgroup $G^\circ\subseteq G$ for which there is an isomorphism of rigid groups $\exp\colon \mg^\circ_{k,\BCH}\isomarrow G^\circ$. It is clear from the construction that this is functorial.
\end{proof}

\begin{Corollary}\label{c:neighbourhood-basis-good-reduction}
	Any rigid group $G$ has a neighbourhood basis  of the identity  $(G_k)_{k\in\N}$ that consists of open subgroups $G_k\subseteq G$ whose underlying rigid space is isomorphic to the closed ball\, $\B^d$ of dimension $d=\dim G$. In particular, each $G_k$ has good reduction.
\end{Corollary}

\begin{proof}
	This is true for $\mathfrak g$, which is isomorphic as a rigid group to $\G_a^d$. By \cref{p:general-exponential}\eqref{p:ge-1}, it then also holds for $G$. The group $G_k$ has good reduction because taking $\O^+$ of $m\colon G_k\times G_k\to G_k$ defines a formal group scheme structure on the unit ball over $\Spf(K^+)$.
\end{proof}

One way to make $\exp$ explicit is to use the following consequence of  \Cref{t:Lie-grp-Lie-alg-correspondence}. 

\begin{Corollary}
	For any rigid group $G$, there exist an $n\geq 0$ and an open subgroup $G^\circ\subseteq G$ that admits a homomorphism of rigid groups $G^\circ\hookrightarrow \GL_n$ that is a locally closed immersion.
\end{Corollary}

\begin{proof}
	By Ado's theorem, there is a faithful representation $\operatorname{Lie} G\to \operatorname{Lie}(\GL_n)$ for some $n$. By \Cref{t:Lie-grp-Lie-alg-correspondence}, it follows that this comes from a morphism $\rho\colon G^\circ\to \GL_n$ for some open subgroup $G^\circ\subseteq G$. It follows from \Cref{p:general-exponential}\eqref{p:ge-2} that $\rho$ is Zariski closed.
\end{proof}

In particular, one can then describe $\exp$ on $G_k$ using the explicit formulas in \cref{ex:exp}.

We now record some more properties of $\exp$ and $\log$ that we do not need in this article, but that fit naturally in the discussion and that we  use in \cite{heuer-sheafified-paCS} in order to compute $R^1\nu_{\ast}G$.

\begin{Lemma}\label{l:exp-on-matrix}
	Let $(R,R^+)$ be any $(K,K^+)$-Banach algebra.
	\begin{enumerate}
		\item\label{l:eom-1}
		If $A,B\in \mathfrak g^\circ(R)$ satisfy $[A,B]=0$, then $\exp(A)$ and $\exp(B)$ commute in $G(R)$ and 
		\[\exp(A+B)=\exp(A)\exp(B).\]
		\item\label{l:eom-2} If $g,h\in G^\circ(R)$ satisfy $gh=hg$, then $[\log(g),\log(h)]=0$ and 
		\[\log(gh)=\log(g)+\log(h).\]
		\item\label{l:eom-3} Let $\mathfrak g_1\subseteq\mathfrak g_2\subseteq \mg^\circ$, and let $G_1\subseteq G_2\subseteq G$ be their images under $\exp$. If $g\in G(R)$ is such that $\ad(g)$ sends $\mathfrak g_{1}$ into $\mathfrak g_{2}$, then conjugation by $g$ sends $G_{1}$ into $G_{2}$ and for any $A\in \mg_{1}(R)$ we have
		\[ \exp(\ad(g)(A))=g^{-1}\exp(A)g.\]
	\end{enumerate}
\end{Lemma}

\begin{proof}
	For the first part, if $[A,B]=0$, then $\mathrm{BCH}(A,B)=A+B$, so this follows from the diagram in \cref{p:general-exponential}.
	
	For part~\eqref{l:eom-3}, we can reduce to the universal situation and replace $R$ by a reduced affinoid $K$-algebra. Then we can check the statement on $K$-points. Here it follows from the functoriality of $\exp$ applied to the conjugation $G\to G$, $h\mapsto g^{-1}hg$.
	
	For part~\eqref{l:eom-2}, we deduce from part~\eqref{l:eom-1} that $\log([p^n]g)=p^n\log(g)$. Since it suffices to prove that $p^n\log(g)$ and $p^n\log(h)$ commute, we may therefore shrink $G^\circ$ and may thus assume by the functoriality of $\exp$ that the restriction of $\ad$ to $G^\circ$ fits into a commutative diagram
	\[\begin{tikzcd}
		G^\circ \arrow[r, "\ad"] \arrow[d, "\log"] & \GL(\mathfrak g)^\circ \arrow[d, "\log"] \\
		\mathfrak g^\circ \arrow[r, "\ad"] & \End(\mathfrak g)^\circ\rlap{.}
	\end{tikzcd}\]
	By part~\eqref{l:eom-3}, we know that $\ad(g)(\log h)=\log(g^{-1}hg)=\log h$. It follows that
	\[ [\log(g),\log(h)]=\log(\ad(g))(\log h)=-\sum_{n=1}^\infty\frac{(1-\ad(g))^n}{n}(\log h)=0,\]
	where the first equality comes from the diagram since the bottom row defines $[-,-]$.
	
	The displayed formula now follows from applying $\log $ to that from part~\eqref{l:eom-1}.
\end{proof}

For $G=\GL_n$, with notation as in \cref{ex:exp}, we slightly more generally have the following. 

\begin{Lemma}\label{l:log-on-matrix}
  Let $A$, $B\in p^{\alpha_0}\m M_n(R^+)$ and $N\in M_n(R)$. The following are equivalent:
  
		\begin{enumerate}
			\item\label{l:lom-1} $AN=NB$, 
			\item \label{i:exp-commnutes}$\exp(A)N=N\exp(B)$.
		\end{enumerate}
\end{Lemma}

\begin{proof}
\eqref{l:lom-1}~$\Rightarrow$~\eqref{i:exp-commnutes} is clear from the formula for $\exp$.  To see \eqref{i:exp-commnutes}~$\Rightarrow$~\eqref{l:lom-1}, we can check the equality on points, which reduces us to the case of $(R,R^+)=(K,K^+)$.
	
	If $N$ is invertible, then $N^{-1}\exp(A)N=\exp(B)\in 1+p^{\alpha_0}\m M_n(R^+)$ and the statement follows from applying $\log$.
	In general,  $\eqref{i:exp-commnutes}$ implies that $\exp(B)$ preserves $\ker N$, which is equivalent to $B$ preserving $\ker N$. In particular, either defines an operator on $\coim(N)$. Similarly, $\exp(A)$ preserves $\im(N)$, which is equivalent to $A$ preserving $\im(N)$. It follows that we can reduce to the  map $N\colon \coim(N)\to \im(N)$, which is an isomorphism.
\end{proof}

\subsection{$\boldsymbol{G}$-torsors in the \'etale and $\boldsymbol{v}$-topology}\label{s:G-torsors}
The theme of this section is that we pass from vector bundles to $G$-torsors. There are several ways to define these, and we shall mainly use the following. For now, we can allow $X$ to be any \'etale-sheafy adic space over a non-Archimedean field $K$ over $\Q_p$.

\begin{Definition}\label{d:G-torsors}
	Consider $X_{\tau}$ for $\tau$ one of $\et$ or $v$. Then by a \emph{$G$-torsor} $P$ on $X_{\tau}$, we mean a \emph{cohomological $G$-torsor} on the site $X_{\tau}$, where $G$ is regarded as a sheaf by sending $Y\in X_\tau$ to the morphisms of adic spaces $Y\to G$ over $K$. Explicitly, $P$ is a sheaf on $X_{\tau}$ with a right action $m\colon P\times G\to P$ by $G$ such that $\tau$-locally on $X$, there is a $G$-equivariant isomorphism $G\isomarrow P$. The morphisms of $G$-bundles are the $G$-equivariant morphisms of sheaves on $X_{\tau}$.
	
	We shall also refer to $G$-torsors on $X_{\et}$ as \emph{\'etale $G$-torsors}, and  to $G$-torsors on $X_{v}$ as \emph{$v$-$G$-torsors}.
	In either case, we say that the torsor is trivial if it is globally on $X_\tau$ isomorphic to $G$. 
	It is clear from this definition that up to isomorphism $G$-torsors on $X_{\tau}$  are classified by the non-abelian sheaf cohomology set 
	$H^1_{\tau}(X,G)$, with the distinguished point corresponding to the trivial torsor.
	
	We can make the analogous definition when $X$ is instead a locally spatial diamond with $X_\tau$ being either $X_\et$ or $X_v$, and where $G$ is the sheaf on $X_\tau$ represented by $G^\diamondsuit$ in $\Dmd_{K,v}$.
\end{Definition}

\begin{Remark}\label{r:GL_n-torsor-vs-vector-torsor}
	For $G=\GL_n$, we recall that there is a natural functor from $\GL_n$-torsors on $X_{\tau}$ to $\tau$-vector bundles (\textit{i.e.}~finite locally free $\O$-modules on $X_\tau$) of rank $n$ that is essentially surjective but not fully faithful: the morphisms of $\GL_n$-torsors are precisely the isomorphisms of vector bundles. In particular, the difference is harmless as long as we are only concerned with isomorphisms, for example in the context of moduli stacks.
\end{Remark}

As usual, there is also a geometric perspective on $G$-torsors. 

\begin{Definition}\label{d:geom-G-torsor} Let $X$ be an \'etale-sheafy adic space over $K$, and let $\tau$ be one of $\et$ or $v$.
	Then a \emph{geometric $G$-torsor} on $X_{\tau}$ is a morphism $E\to X$ of sheaves on $X_\tau$ with a right $G$-action $E\times G\to E$ over $X$ such that there is a cover $X'\to X$ in $X_\tau$ with a $G$-equivariant isomorphism $X'\times G\isomarrow X'\times_XG$ over $X'$. The morphisms of geometric $G$-torsors are the $G$-equivariant morphisms of sheaves on $X_\tau$.
	
	We can make the analogous definition when $X$ is instead a locally spatial diamond, where $X_\tau$ is either $X_\et$ or $X_v$, and $G$ is seen as the sheaf on $X_\tau$ represented by $G^\diamondsuit$ in $\Dmd_{K,v}$.
\end{Definition}

\begin{Remark}\leavevmode
	\begin{enumerate}
		\item We caution that both in \Cref{d:G-torsors} and in~\Cref{d:geom-G-torsor}, depending on the space $X$, it can in general make a difference whether we consider $G$-torsors on $X_\et$ or $G$-torsors on $X^\diamondsuit_{\et}$, because in general the structure sheaves are not identified via the equivalence $X_\et=X^\diamondsuit_{\et}$. For example, for the rigid space $X=\Spa(K[X]/X^2)$, the automorphisms of the trivial $\G_a$-torsor on $X_\et$ are given by $\O_{X_\et}(X)=K[X]/X^2$, whereas the automorphisms of the trivial $\G_a$-torsor on $X^\diamondsuit_\et$ are given by $\O_{X^\diamondsuit_\et}(X^\diamondsuit)=K$ because perfectoid spaces do not see the non-reduced structure.
		
	\item While it may be true for an adic space $X$ that geometric $G$-torsors on $X_{\et}$ are represented by adic spaces, it is definitely not the case that  geometric $G$-torsors on $X_{v}$ are always represented by adic spaces. However, we will see in \cref{c:G-torsor-is-diamond} that they are still always diamonds, which is not immediately obvious from the definition. In particular, when $X$ is a locally spatial diamond and $\tau=v$, one gets an equivalent definition if one takes $E\to X$ to be a morphism of diamonds instead of $v$-sheaves.
	\end{enumerate}
\end{Remark}

\begin{Proposition}\label{p:fully-faithful-G-torsors-et-to-v}\leavevmode
	\begin{enumerate}
		\item\label{p:ffGtetv-1} Let $X$  be an \'etale-sheafy adic space over $K$, or let $X$ be a locally spatial diamond.
	The functor sending a geometric $G$-torsor $E\to X$ to its sheaf of sections $E\leftarrow X$ defines an equivalence of categories
	\[\{\text{geometric $G$-torsors on $X_{\tau}$}\}\longisomarrow\{\text{cohomological $G$-torsors on $X_{\tau}$}\}.\]
	Any morphism in either category is an isomorphism.
	\item\label{p:ffGtetv-2} When $X$ is a locally spatial diamond, there is a natural fully faithful functor
	\[ \{\text{$G$-torsors on $X_{\et}$}\}\longhookrightarrow\{\text{$G$-torsors on $X_{v}$}\}\]
	given by sending a $G$-torsor on $X_{\et}$ to the $v$-sheaf of sections of its associated geometric $G$-torsor. On isomorphism classes, this functor is given by the natural map
        \[H^1_{\et}(X,G)\lra H^1_{v}(X,G).\]
	\item\label{p:ffGtetv-3} When $X$ is an \'etale-sheafy adic space, diamondification defines a functor
	\[ \{\text{$G$-torsors on $X_{\et}$}\}\lra\left\{\text{$G^\diamondsuit$-torsors on $X_{\et}^\diamondsuit$}\right\}.\]
	\item\label{p:ffGtetv-4} Let $X$ be an \'etale-sheafy adic space such that $\nu\colon X_v\to X_\et$ satisfies $\nu_{\ast}\O=\O$; for example, this holds when~$X$ is perfectoid or a semi-normal rigid space. Then the functor from~\eqref{p:ffGtetv-3} is an equivalence. In particular, we then have a fully faithful functor
	\[\{\text{$G$-torsors on $X_{\et}$}\}\longhookrightarrow\{\text{$G$-torsors on $X_{v}$}\}.\]
	\end{enumerate}
\end{Proposition}

As one can already see  from the case of $G=\G_m$  discussed in \cite{heuer-v_lb_rigid}, the functor from \Cref{p:fully-faithful-G-torsors-et-to-v}\eqref{p:ffGtetv-4} is in general far from being essentially surjective when $X$ is a rigid space.

\begin{proof}
  \eqref{p:ffGtetv-1}~ It is clear that both cohomological $G$-torsors on $X_{\tau}$ and geometric $G$-torsors  on $X_{\tau}$ satisfy $\tau$-descent. It therefore suffices to see  that the endomorphisms of the trivial object are identified via the functor: for the trivial cohomological $G$-torsor, it is clear that the $G$-equivariant morphisms $G\to G$ over~$X$ correspond to $g\in G(X)$ via the map $h\mapsto gh$. Similarly, for the trivial geometric $G$-torsor $X\times G$, any $G$-equivariant morphism $\phi\colon X\times G\to X\times G$ of $\tau$-sheaves over $X$ is uniquely determined by the map $X\xrightarrow{\id,1}X\times G\xrightarrow{\phi} X\times G\xrightarrow{\pi_2} G$; hence the endomorphisms of the trivial geometric $G$-torsor are also identified with $G(X)$.
  
\eqref{p:ffGtetv-2}~  It is clear from \eqref{p:ffGtetv-1} that the functor is well defined. The claim that this is fully faithful can be checked locally, so we can again reduce to the trivial torsor. Here it follows from the fact that for $\mu\colon X_v^\diamondsuit\to  X_\et^{\diamondsuit}$, we have $\mu_{\ast}G=G$ by definition.

\eqref{p:ffGtetv-3}~ This is clear from interpreting both sides as geometric $G$-torsors.

\eqref{p:ffGtetv-4}~ The condition $\nu_{\ast}\O=\O$ guarantees that for any $Y\in X_\et$, the morphisms $Y\to G$ of adic spaces over~$K$ are in bijection with the morphisms $Y^\diamondsuit\to G^\diamondsuit$ of diamonds over $K$. The statement then follows from the identification $X_\et=X_\et^\diamondsuit$.
\end{proof}

\begin{Remark}
	If $G$ is a linear algebraic group, there is a third perspective on $G$-torsors: the Tannakian point of view. We sketch it here and refer to \cite[Appendix to Lecture 19]{ScholzeBerkeleyLectureNotes} for details. A \emph{Tannakian $G$-torsor} on a sousperfectoid space $X$ is an exact tensor functor 
	\[P\colon\Rep_G\lra \mathrm{Bun}_{X,\et},\]
	 where $\Rep_G$ is the category of algebraic representations $G\to \GL(V)$ of $G$ considered as an algebraic group scheme, and $\mathrm{Bun}_{X,\et}$ is the category of \'etale vector bundles on $X$. By \cite[Theorem 19.5.1]{ScholzeBerkeleyLectureNotes}, the categories of $G$-torsors  and Tannakian $G$-torsors on $X_{\et}$ are equivalent via the functor that sends a $G$-torsor $P$ to the Tannakian $G$-torsor  $V\mapsto P\times^GV$.
	 
	 One can deduce using Kedlaya--Liu's \cref{t:loc-free-O-modules} that on a perfectoid space $X$, the categories of $G$-torsors on $X_{\et}$ and $G$-torsors on $X_{v}$ are equivalent for linear algebraic $G$.
	 
	 However, we do not know if the case of $G^+:=G(\O^+)$ for linear algebraic groups over $K^+$ could be deduced in a similar way from that of $\GL_n(\O^+)$ discussed in \cref{s:generalised-representations}:  the Tannakian approach works via algebraisation, but for $G^+$-torsors it seems less clear when these come from $G^+$-torsors on $\Spec(R^+)$, already for $G=\GL_n$.
\end{Remark}

\section{Reduction of structure group to open subgroups}\label{s:quot-sheaves}
In the past section, we have seen using the exponential that any rigid group $G$ has a neighbourhood basis of open subgroups. In this section, we study when $G$-torsors in the $v$-topology admit a reduction of structure group to open subgroups. We then deduce the main theorem, Theorem~\ref{t:G-torsors-on-perfectoid-spaces}.

\subsection{Approximation for quotient sheaves}\label{s:approx-quot-sheaves}
The key technical result of this section is a generalisation of the approximation property \cref{p:GL_n(O^+/p^n)-in-proet-site} for $\GL_n(\O^+/\varpi)$ to all rigid groups and all rigid open subgroups.  

\begin{Proposition}\label{p:approx-property-for-G/U}
	Let $G$ be a rigid group, and let $U\subseteq G$ be any rigid open subgroup $($not necessarily normal\,$)$. Then the sheaf of cosets $G/U$ on $\mathrm{Sous}_{K,\et}$ satisfies the approximation property of \cref{d:approx-property}: for any affinoid perfectoid tilde-limit $X\approx \varprojlim_{i\in I} X_i$, we have
	\[ G/U(X)=\varinjlim_{i\in I} G/U(X_i).\]
	In particular, $G/U$ is already a $v$-sheaf on $\mathrm{Sous}_{K}$. 
\end{Proposition}

\begin{proof}
	For any $i\in I$ and any $W_i\to X_{i}$ in $X_{i,\etqcqs}$, let $W:=W_i\times_{X_i}X$ and $W_j:=W_i\times_{X_i}X_j$ for $j\geq i$ be the pullbacks. Then we obtain a natural pullback map
	\[\varinjlim_{j\geq i} G(W_j)/U(W_j)\lra G(W)/U(W). \]
	Using the identification
	$ X_{\etqcqs}=2\text{-}\varinjlim_{i\in I}X_{i,\etqcqs}$ from \cref{l:tilde-lim-2-lim-on-etqcqs},  for varying $i$ and $W_i$, we can regard this as a morphism of presheaves on $X_{\etqcqs}$. We consider its sheafification: if we denote by $\mu_i\colon X_{\etqcqs}\to X_{i,\etqcqs}$ the projections, then this can be described as the natural map
	\[\phi\colon \varinjlim_{i\in I}\mu_i^{\ast}(G/U)\lra G/U.\]
	By \cite[Tag~09YN]{StacksProject}, the left-hand side satisfies 
	\[ \varinjlim_{i\in I}\mu_i^{\ast}(G/U)(X)=\varinjlim_{i\in I} G/U(X_i).\]
	Hence it suffices to prove that $\phi$ is an isomorphism.
        
	\begin{Claim}
		The map $\phi$ is injective.
	\end{Claim}

        \begin{proof}
	If $g_1,g_2\in G(X_i)$ have the same image in $G(X)/U(X)$, then $\delta=g_1^{-1}\cdot g_2\in G(X_i)$ lands in $U(X)\subseteq G(X)$. To deduce that $\delta$ lands in $U(X_j)$ for some $j\geq i$, we no longer need the group structure on $G$ and use the following argument that we learnt from \cite[Lemma 6.13(iv)]{perfectoid-spaces}. 

        \begin{Claim}
		Let $H$ be an affinoid adic space over $K$ and $V\subseteq H$ a rational subspace. Let $X_i\to H$ be a morphism such that $X\to X_i\to H$ factors through $V$. Then $X_j\to X_i\to H$ factors through $V$ for $j\gg i$.
	\end{Claim}

        \begin{proof} For $j\geq i$, let $W_j$ be the pullback of $V$ along $X_j\to H$; then $W_j\subseteq X_j$ is still rational open, hence affinoid. The assumptions imply that $X\to X_j$ factors through $W_j$ and induce a homeomorphism $|X|\to \varprojlim_{j\geq i} |W_j|$. We now use that $(W_j)_{j\geq i}$ is an inverse system of affinoid adic spaces; hence $(|W_j|)_{j\geq i}$ is an inverse system of spectral spaces with spectral transition maps. Since $\varprojlim_{j\geq i} |X_j|\backslash |W_j|$ is empty, it follows from \cite[Tag 0A2W]{StacksProject} that $X_j\backslash W_j$ is empty for some $j\geq i$. Hence $X_j=W_j$, and thus $X_j\to H$ factors through $V$. 
	\end{proof}

        We apply the claim locally on $X$: let $H$ be any affinoid open subspace of $G$, and cover $U\cap H$ by rational opens $V$. As $X_i$ is quasi-compact, finitely many such opens cover the image of $X_i\to U$. Replacing $X$ and the $X_i$ by the respective pullbacks of $V\subseteq H$, we may assume that $\delta$ lies in $V(X)$. The claim shows that we already have $\delta\in U(X_j)$ for some $j\geq i$.
	 \end{proof}

        It thus remains to prove that $\phi$ is surjective. For this we again localise on $G$: let $\Spa(A,A^+)=V\subseteq G$ be any affinoid open such that $(A,A^+)$ is of topologically finite type over $(K,K^+)$. By \cite[Lemma~3.3(ii)]{Huber-ageneralisation}, this means that there is a subring of definition $A_0\subseteq A^+$ of topologically finite type (hence automatically of topologically finite presentation) over $K^+$ such that $A^+$ is the integral closure of $A_0$ in $A$.
 	 Write $X=\Spa(R,R^+)$ and $X_i=\Spa(R_i,R^+_i)$; then any point $x\in G(X)$ that factors through $V$ defines a map $f\colon A_0\to R^+$. By \cite[Lemma~3.10]{heuer-diamantine-Picard}, the assumption $X\approx \varprojlim X_i$ implies $R^+/p^n=\varinjlim R^+_i/p^n$ for all $n\in \N$. Since $A_0/p^n$ is of finite presentation over $K^+/p$, the reduction of $f$ mod $p^n$ factors through a morphism
 	\[ A_0/p^n\lra R_i^+/p^n.\]
 	for some $i\in I$. We now use the following. 
        
 	\begin{Lemma}\label{lem44}
 		Let $A$ be a  $p$-adically complete flat $K^+$-algebra such that $A\tf$ is a smooth affinoid $K$-algebra. Let us assume for simplicity that $\Omega^1_{A[\frac{1}{p}]|K}$ is finite free. Then there is a $t\geq 0$ such that for  any $p$-adically complete $K^+$-algebra $R$ and any homomorphism 
 		of\, $K^+$-algebras
 		\[f_t\colon A/p^n\lra R/p^n\]
 		for some $n>t$, there is a homomorphism of\, $K^+$-algebras $f\colon A\to R$ with $f\equiv f_t\bmod p^{n-t}$.
 	\end{Lemma}
        
 	\begin{proof}
 		This is a statement about torsion in the cotangent complex: by \cite[Proposition~III.2.2.4]{IllusieCotangent},  for any $s\leq n\leq 2s$ in $\N$ and any morphism
 		$f_n\colon A\to R/p^n$, there is 
 		 an obstruction class 
 		\[o_{n,s}\in \Ext_{n,s}^1:=\Ext^1_{R/p^n}\left(L_{A/p^n|K^+/p^n}\otimes_{A/p^n}^{\mathbb L}R/p^n,R/p^s\right)\]
 		that vanishes if and only $f_n$ lifts to a morphism $A\to R/p^{n+s}$.
                
 		\begin{Claim}
 			There is a $t\in \N$ independent of $s,n$ such that $\Ext_{n,s}^1$ is $p^t$-torsion.
 		\end{Claim}

                \begin{proof}
 			We first note that the inclusion $A\subseteq A\tf^\circ$ has bounded $p$-torsion cokernel, and similarly for $K^+\to \O_K:=K^\circ$. Therefore, without loss of generality, we may replace $A$ by a ring of integral elements that is of topologically finite presentation over $\O_K$ and $R$ by $R\hat{\otimes}_{K^+}\O_K$, so that we  may assume that  $K^+=\O_K$.
 			In this setting, we can use the analytic cotangent complex for formal schemes and rigid spaces introduced by Gabber--Ramero \cite[Section~7]{GabberRamero}: this is a pseudo-coherent complex $L^{\an}_{ A|\O_K}$ of $A$-modules such that $H^0(L^{\an}_{ A|\O_K})=\Omega_{A|\O_K}$ and \[L^{\an}_{ A|\O_K}\otimes^{\mathbb L}_{A}A/p^k=L_{ A|\O_K}\otimes^{\mathbb L}_{A}A/p^k\]
 			for all $k\in \N$. Moreover, 
 			\[L^{\an}_{ A|\O_K}\tf=L^{\an}_{ A[\frac{1}{p}]|K}=\Omega_{A[\frac{1}{p}]|K}\]
 			due to the assumption that $A[\tfrac{1}{p}]|K$ is smooth. 
 			Let $\Omega^+$ be any finite free $A$-sublattice of $\Omega_{A[\frac{1}{p}]|K}$ that contains the image of $ \Omega_{A|\O_K}\to  \Omega_{A[\frac{1}{p}]|K}$. Then it follows that the cone $C:=\mathrm{cone}(h)$ of the canonical map $h\colon L^{\an}_{ A|\O_K}\to \Omega_{A|\O_K}\to \Omega^+$ is exact after the inversion of $p$. Thus by the same argument as in the proof of \cite[Proposition 6.10(iii)]{perfectoid-spaces}, it follows that for all $i\leq 0$, the cohomology $H^i(C)$ is killed by $p^m$ for some $m$ depending on $i$. As $C$ is bounded above, we may therefore choose $t$ such that $H^i(C)$ is killed by $p^t$ for all $i\geq -2$.
 			
 			We now first apply $\otimes^{\mathbb L}_{A}R/p^n$ and then $\Ext^1_{R/p^s}(-,R/p^s)$ to the distinguished triangle
 			\[ C\lra L^{\an}_{ A|\O_K}\lra \Omega^+\]
 			and obtain an exact sequence
 			\[ \Ext^1\left(C\otimes^{\mathbb L}_{A}R/p^n,R/p^s\right)\lra \Ext^1_{n,s}\lra \Ext^1\left(\Omega^+\otimes_{A}R/p^n,R/p^s\right).\]
 			The last term vanishes because $\Omega^+$ is a finite free $A$-module by construction.
 			The first term is $p^t$-torsion because the second argument of $\Ext^1$ is concentrated in degree $0$, so only the first two terms of $C$ contribute, and the truncation $\tau_{\geq -2}C$ is killed by $p^t$.
 		\end{proof}
                
 		Now let  $s,n$ be such that $t<s\leq n< 2s$. It is clear from short exact sequences  that
 		\[ \Ext_{n,s}^1\overset{p^t}\lra\Ext_{n-t,s}^1\]
 		sends the obstruction class $o_{n,s}$ to $o_{n-t,s}$. By the claim, this morphism is zero. It follows that we may lift the reduction $f_{n-t}\colon A\xrightarrow{f_n} R/p^n\to R/p^{n-t}$ to a morphism $A\to R/p^{n+s-t}$. As $s>t$, we see inductively that we obtain a lift of $f_{n-t}$ to the complete $\O_K$-algebra $R$.
 	\end{proof}
        
	Since after localising we may assume that $\Omega_{A|K}$ is finite free, Lemma~\ref{lem44}  applies, and setting $k=n-t$, we see that there is a sequence of indices $(i_k)_{k\in \N}$ and morphisms $f_k\colon A_0\to R^+_{i_k}$ such that $f_k\equiv f\bmod p^k$. Since the ring extension $A_0\to A^+$ is integral with bounded cokernel, it is clear that $f_k$ extends to $f_k\colon A^+\to R^+_{i_k}$, and that after re-indexing we may still assume that $f_k\equiv f\bmod p^k$. After the inversion of $p$, the $f_k$ define a sequence of points in $G(X_{i_k})$. Let $x_{k}\in G(X)$ be the images of these points under $G(X_{i_k})\to G(X)$.
        
	\begin{Claim}\label{cl:diff-in-U(X)}
		For $k\gg 0$, the difference $\delta_k:=x^{-1}x_k$ lies in $U(X)$.
	\end{Claim}

        This will prove that
	\[ \varinjlim_{i\in I} G(X_i)\lra G(X)/U(X)\]
	is surjective (generalising \cite[Lemma~3.10]{heuer-diamantine-Picard}, which is the case of $G=\O^+$, $U=\varpi\O^+$). This will in turn finish the proof of the approximation property. That $G/U$ is a $v$-sheaf then follows from \cref{p:approx-property-implies-et-v}.

        \begin{proof}
		Let $Z_k\subseteq X$ be the preimage of $U\subseteq G$ under the map
		$\delta_k\colon X\to G$.
		It is clear that $Z_k$ is open. It therefore remains to prove that each $z\in X$ is contained in $Z_k$ for some $k$. From this the claim follows by the compactness of $X$.
		
		Let $z\in X$ be any point, and let $\Spa(C,C^+)\to X$ be any morphism, where $(C,C^+)$ is a perfectoid field such that $z$ is in the image. The points $x$ and $x_k$ both define maps
		$\Spa(C,C^+)\to X\to V$
		that correspond to two homomorphisms of $K^+$-algebras
		\[ A^+\lra R^+\lra C^+.\]
		By construction, these agree modulo $p^k$. 
		We now consider the affinoid rigid space $V_C:=V\times_{\Spa(K)}\Spa(C)$ over $\Spa(C)$ and the topological space $V_C(C)$. Let $W_k\subseteq V_C(C)$ be the subspace of points whose associated morphisms $A^+\hat{\otimes}_{K^+}C^+\to C^+$ agree with $x$ mod $p^k$.
                
		\begin{Claim}
		Any open neighbourhood of $x\in V_C(C)$ contains one of the $W_k$.
		\end{Claim}

                \begin{proof}
		 We reduce to the case of $\B^d$, where the statement is clear: for this reduction, embed $V\subseteq \B^d$ into some $d$-dimensional ball corresponding to a morphism 
		 $K^+\langle T_1,\dots,T_n\rangle\to A^+$ with bounded $p$-torsion cokernel; then  the compositions 
		 \[K^+\langle T_1,\dots,T_n\rangle\lra A^+\lra A^+/p^k\lra C^+/p^k\] also agree mod $p^k$ for any two points in $W_k$. The claim follows since $V_C(C)\subseteq \B^d(C)$ carries the subspace topology.
		\end{proof}
                
		For the base-change $U_C\subseteq G_C$ of $U\subseteq G$ to $C$, this means that the open neighbourhood $(x\cdot U_C(C))\cap V_C(C)$ of $x$ contains  $W_k$ for $k\gg 0$. Thus $x_k\in W_k\subseteq  x\cdot U_C(C)$, showing $\delta_k(z)=x^{-1}x_k\in U_C(C)$, whence $z\in Z_k$.
	\end{proof}
        
This finishes the proof of \cref{p:approx-property-for-G/U}.
\end{proof}

\begin{Proposition}\label{p:reduction-of-structure-group-to-open-subgroup}
	Let $G$ be a rigid group over $K$, and let $U\subseteq G$ be a rigid open subgroup. Let $X$ be a sousperfectoid adic space over $K$, and let $\nu\colon X_v\to X_{\et}$ be the natural morphism of sites. Then the natural map
	\[ R^1\nu_{\ast}U\lra R^1\nu_{\ast}G\]
	is surjective. In other words, any $G$-torsor $E$ on $X_v$ admits a reduction of structure group to $U$ \'etale-locally on~$X$; \textit{i.e.}~there exist an \'etale cover $Y\to X$ and a $U$-torsor $F$ on $Y_v$ such that $F\times^UG=E\times_XY$ as $G$-torsors. If\, $G$ is commutative, then we more generally have 
	\[R^k\nu_{\ast}U= R^k\nu_{\ast}G\]
	and $R^k\nu_{\ast}(G/U)=1$ for all $k\geq 1$. In particular, in this case $F$ is unique up to isomorphism.
\end{Proposition}

\begin{proof}
	We first observe that if $U\subseteq G$ is normal, then we can consider the short exact sequence of sheaves of pointed sets
	\[ R^1\nu_{\ast}U\lra  R^1\nu_{\ast}G\lra  R^1\nu_{\ast}(G/U),\]
	and the result follows from the fact that by \cref{p:approx-property-for-G/U}, we may apply \cref{p:approx-property-implies-et-v}\eqref{p:apiev-2} to $G/U$ and conclude that  $R^1\nu_{\ast}(G/U)=1$. If $G$ is commutative, we more generally have $R^k\nu_{\ast}(G/U)=1$ by \cref{p:approx-property-implies-et-v}\eqref{p:apiev-3}, which gives the result in the commutative case.
	
	In general, $G/U$ is only a sheaf of cosets, but we can still make sense of the change-of-fibre
	\[\overline{E}:=E\times^GG/U:=G\backslash(E\times G/U), \]
	where the last term is the quotient of $E\times G/U$ with respect to the left $G$-action via $g\cdot (e,x):=(eg^{-1},gx)$. Then $\overline{E}$ is a fibre bundle on $X_v$ with structure group $G$ and fibres $G/U$. For this we can show the following.

	       \begin{Lemma}\label{l:section-of-fibre-bundle}
			Let $X$ be a locally spatial diamond, and let $E$ be a $G$-torsor on $X_v$. Then the following are equivalent: 
			\begin{enumerate}
				\item\label{i:4.9.1} $E$ admits a reduction of structure group to $U$.
				\item\label{i:4.9.3} $\overline{E}$ admits a section $s\in \overline{E}(X)$.
			\end{enumerate}
			When $U\subseteq G$ is a normal subgroup, these are furthermore equivalent to 
			
			\begin{enumerate}[resume]
			\item\label{i:4.9.2} $\overline{E}$ is isomorphic to $X\times G/U$.
			\end{enumerate}
		\end{Lemma}

	\begin{proof}
			To see the implication 	\eqref{i:4.9.1} $\Rightarrow$ \eqref{i:4.9.3}, let $F$ be a $U$-torsor such that $E=F\times^UG$. Then
			\[ \overline{E}=F\times^UG\times^GG/U=F\times^UG/U=U\backslash(F\times G/U).\]
			We claim that the composition
			\[ t\colon F\lra E\lra \overline{E}\]
			is constant in the sense that it factors through an $X$-point $s\colon X\to \overline{E}$. To see this, it suffices to see that there is a basis of $X_v$ consisting of spaces $Z$ such that the image of $t(Z)$ is a single element. But the above description of $\overline{E}$ shows that for any $Z$ on which $F\simeq U$ is trivial, $t$ becomes isomorphic to
			\[ U\lra G\lra G/U,\]
			which  has image $1$. Hence $t$ is locally constant, and therefore constant, as we wanted to see.

			 To see \eqref{i:4.9.3} $\Rightarrow$ \eqref{i:4.9.1}, let $F\subseteq E$ be the subsheaf defined as the fibre product 
			\[\begin{tikzcd}
				F \arrow[d] \arrow[r] & E \arrow[d] \\
				X \arrow[r,"s"] & \overline{E}
			\end{tikzcd}\]
			in $v$-sheaves on $X_v$. We claim that $F$	is a reduction of structure group of $E$ to $U$. To see this, we first observe that the right action of $U$ on $E$ leaves the reduction map $E\to \overline{E}$ invariant. It follows that the right $G$-action on $E$ restricts to a right $U$-action on $F$. 
			
			Now let  $f\colon Y\to X$ be a $v$-cover by a strictly totally disconnected space on which $E$ admits a trivialisation $\gamma\colon Y\times_XE=Y\times G$. This induces an isomorphism $\overline{\gamma}\colon Y\times_X\overline{E}=Y\times G/U$ that we can use to regard $f^\ast s$ as a section of $G/U(Y)$. Using that $G/U(Y)=G(Y)/U(Y)$ since $Y$ is strictly totally disconnected, we can now compose $\gamma$ with an element in $G(Y)$ to ensure that $\overline{\gamma}(s)=1$.
			
			Then via $\gamma$, the pullback of the above diagram along $Y\to X$ is isomorphic to
			\[\begin{tikzcd}
				Y\times_XF \arrow[d] \arrow[r] & Y\times G \arrow[d] \\
				Y \arrow[r,"{(\id,1)}"] &  Y\times G/U\rlap{.}
			\end{tikzcd} \]
			It follows that we have a $U$-equivariant isomorphism
			\[ Y\times_XF\longisomarrow Y\times U\,\]
			as the right-hand side is also the fibre product. Thus $F$ is a $v$-topological $U$-torsor on $Y$.
			
			Since the natural map $F\to E$ is $U$-equivariant by construction, it now follows from the fact that any morphism between $G$-torsors is an equivalence (see \Cref{p:fully-faithful-G-torsors-et-to-v}\eqref{p:ffGtetv-1}) that the map $F\times^UG\isomarrow E$ is an isomorphism; hence $F$ is a reduction of structure group of $E$.
	\end{proof}
        
		It thus suffice to prove the following. 

	        \begin{Lemma}\label{l:section-over-et-cover}
			Let $X$ be a sousperfectoid adic space, and let $E$ be a $v$-$G$-torsor on $X$. Then
			there is an \'etale cover $Y\to X$ over which $\overline{E}$ admits a section.
		\end{Lemma}

                \begin{proof}
		  Without loss of generality, we may replace $U$ by a smaller rigid open subgroup. By \cref{c:neighbourhood-basis-good-reduction}, we may therefore assume that the adic space underlying $U$ is isomorphic to the unit ball $\B^d$. In this case, we can use an argument similar to that in the proof of \cref{p:approx-property-implies-et-v}.
		  
		We first deal with the case that $X$ is strictly totally disconnected. Let $Y\to X$ be a $v$-cover by an affinoid perfectoid space on which there is an isomorphism $\alpha\colon  G\times Y\isomarrow E\times_XY$. This induces a class $\psi\in G(Y\times_XY)$ such that $\overline{E}$ can be described on any $Z\in X_v$ as the equaliser
		\[\begin{tikzcd}
			\overline{E}(Z)& G/U(Y\times_XZ)&  G/U(Y\times_XY\times_XZ).
			\arrow["\psi\cdot\pi_1^{\ast}", shift left=1, from=1-2, to=1-3]
			\arrow["\pi_2^\ast"', shift right=1, from=1-2, to=1-3]
			\arrow[from=1-1, to=1-2]
		\end{tikzcd}\]
		Approximating $Y\to X$ by sousperfectoid spaces $Y_i\hookrightarrow \B^n\times X$ as in \cref{l:v-cover-approximate-by-balls}, and using that we have $Y_{/X}^{\times 2}\approx \varprojlim_{i\in I} Y_{i/X}^{\times 2}$ by \cref{l:product-of-tilde-limits}, it follows by the approximation property for $G/U$ from \cref{p:approx-property-for-G/U} that 
		\[ \overline{E}(Y)=\varinjlim_{i\in I} \overline{E}(Y_i).\]
		This shows that the composition
		\[ Y\xrightarrow{1,\id}G/U\times Y\overset{\alpha}\lra\overline{E}\times_XY\lra \overline{E}\]
		factors through a morphism $y\colon Y_i\to \overline{E}$ for some $i$. 
		Since $Y_i\to X$ is split, we thus obtain a section $X\to \overline{E}$.
		This proves the lemma for strictly totally disconnected $X$.
		
		The general case follows from this by a similar approximation argument: let $X$ be any affinoid sousperfectoid space, and let $Y\to X$ be the quasi-pro-\'etale cover $Y\to X$ by a strictly totally disconnected space from \cref{l:proet-Colmez-cover}; \textit{i.e.}~$Y\approx \varprojlim Y_i$ is a cofiltered tilde-limit of \'etale surjective maps $Y_i\to Y$. By the first part, there is an isomorphism $\alpha\colon G/U\times Y\isomarrow \overline{E}\times_XY$. Again by the approximation property, we have
		$\overline{E}(Y)=\varinjlim \overline{E}(Y_i)$.
	        \end{proof}
                
	This finishes the proof of \cref{p:reduction-of-structure-group-to-open-subgroup}
\end{proof}

As a further application, the strategy of this section also gives  the following useful results. 

\begin{Proposition}\label{p:reduction-of-strgrp-in-tilde-limits}
	Let $G$ be a rigid group, and let $U\subseteq G$ be any open subgroup.
	Let $X$ be a locally spatial diamond, and let $E$ be a $G$-torsor on $X_v$. Let $Y=\varprojlim_{i\in I} Y_i$ be a limit of locally spatial diamonds with qcqs transition maps and compatible maps $Y_i\to X$. If\, $E\times_XY$ has a reduction of structure group to $U$, then there is an $i\in I$ such that $E\times_XY_i$ has a reduction of structure group to $U$.
\end{Proposition}

\begin{proof}
	By \Cref{l:section-of-fibre-bundle}, it suffices to prove that we have
	\[ \overline{E}(Y)=\varinjlim_i \overline E(Y_i).\]
	To see this, let $\wt X\to X$ be any $v$-cover by a perfectoid space on which $E$ becomes trivial, and consider the fibre products $\wt Y:=Y\times_X\wt X$ and $\wt Y_i:=Y_i\times_X\wt X$ in the category of locally spatial diamonds (these exist by \cite[Corollary 11.29]{etale-cohomology-of-diamonds}). Then since fibre products commute with limits, we have $\wt Y=\varprojlim \wt Y_i$. By \Cref{p:approx-property-for-G/U} and  \Cref{l:approx-prop-for-fibre-products},   it follows that we have
	\[ G/U\left(\wt Y\right)=\varinjlim_i G/U\left(\wt Y_i\right),\]
	and similarly for $\wt Y\times_Y\wt Y$.
	Since $\overline{E}\simeq G/U$ over $\wt X$, and hence also over each $\wt Y_i$ and $\wt Y$,  it follows from this that
	\[ \overline{E}(Y)=\cH^0\left(\wt Y\to Y,\overline{E}\right)=\varinjlim_i \cH^0\left(\wt Y_i\to Y_i,\overline{E}\right)=\varinjlim_i\overline{E}(Y_i),\]
	as we wanted to see.
\end{proof}

\begin{Corollary}
	Let $X$ be any locally spatial diamond. Let $S$ be any adic space over $K$. Let $E$ be a $G$-torsor on $(X\times S)_v$. Let $\eta\colon \Spa(C,C^+)\to S$ be any morphism where $C$ is a perfectoid field, and assume that the pullback of\, $E$ to $X\times \eta$ admits a reduction of structure group to $U$. Then there is an \'etale map $S'\to S$ whose open image contains $\im(\eta)$ such that $E$ admits a reduction of structure group to $U$ over $X\times S'$.
\end{Corollary}

\begin{proof}
	We first assume that $S$ is strictly totally disconnected; then by \Cref{l:v-cover-approximate-by-balls}, we can choose an approximation of $\eta$ of the form $\Spa(C,C^+)\approx \varprojlim S_i\to S$, where each $S_i\to S$ is split. Then $X\times \eta= \varprojlim X\times S_i$ in locally spatial diamonds; hence \Cref{p:reduction-of-strgrp-in-tilde-limits} shows that $V$ admits a reduction of structure group to $U$ on some $X\times S_i$. Since $S_i\to S$ is smooth, its image is an open in $S$, which is again strictly totally disconnected. Hence the map is split over its open image, so we obtain an open of $S$ with the desired property.
	
	We deduce the general case by considering the pro-\'etale strictly totally disconnected cover $\widetilde{S}\approx\varprojlim_i S_i\to S$ of \Cref{l:proet-Colmez-cover}: the map $\eta$ then factors through a map $\widetilde{\eta}\colon \Spa(C,C^+)\to \widetilde{S}$. By the first part, there is an open $W\subseteq \widetilde{S}$ such that $E$ admits a reduction of structure group to $U$ over $X\times W$. After shrinking $W$, we may, without loss of generality, assume that it comes via pullback from some open $W_i\subseteq S_i$. Applying \Cref{p:reduction-of-strgrp-in-tilde-limits} once more, this shows that $E$ becomes trivial over $X\times W_i$ for some open $W_i\subseteq S_i$.
\end{proof}

\subsection{$\boldsymbol{p}$-adic integral subgroups for rigid groups of good reduction }\label{s:integral-G}
In this subsection, we take a closer look at the structure of rigid open subgroups in the special case that~$G$ has good reduction.
Throughout this subsection, we fix a rigid group $G$ of good reduction as well as a smooth formal model $\mathcal G$ of $G$ over $\Spf(K^+)$. The Lie algebra $\mathfrak g^+$ of $\mathcal G$ is a finite projective $K^+$-module, hence free, and we can consider it as a canonical open $K^+$-lattice $\mathfrak g^+\subseteq \mathfrak g$. As before, we also consider these as $v$-sheaves on $K_v$ and write explicitly  $\mathfrak g^+(K)$ and $\mathfrak g(K)$ for the underlying  modules. 

\begin{Lemma}
	The sheaf\, $G$ on $\Sous_{K,\et}$ is the analytic sheafification of the functor 
	\[ (R,R^+)\longmapsto \mG(R^+)\]
\end{Lemma}

\begin{proof}
	For any formal affine open $\Spf A\subseteq \mG$, the adic generic fibre is by definition given by $(\Spf A)^{\ad}_{\eta}=\Spa(A\tf,A^+)$, where $A^+$ is the integral closure of $A$ in $A\tf$. Hence the points $G(R,R^+)$ correspond to the homomorphisms of $K^+$-algebras $A\to R^+$.
\end{proof}

We can use this to give a natural generalisation of the sheaves $\GL_n(\O^+/p^k\m)$. 

\begin{Definition}
	For any $0\leq k\in \log|K|$, consider the  sheaf $\overline G_{k}$ on the big \'etale site of sousperfectoid spaces $\mathrm{Sous}_{K,\et}$ defined by \'etale sheafification of the presheaf on $\Sous_{K,\et}$
		\[ Y\longmapsto \mG\left(\O^+/p^k\m\O^+(Y)\right).\]
	We denote by $G_{k}\subseteq G$ the kernel of the morphism $G\to \overline G_{k}$.
\end{Definition}

For $k=0$, the restriction to $\Perf_K$  of $\overline G_{k}$ is the sheaf $\overline G^\diamondsuit$ studied in \cite[Section~5]{heuer-Picard-good-reduction}. 

\begin{Lemma}\label{l:ses-of-G_k}
	The following sequence on $\Sous_{K,\et}$ is short exact:
	\[ 0\lra G_k\lra G\lra \overline G_{k}\lra 0.\]
\end{Lemma}

\begin{proof}
	The left exactness is clear by definition. The sequence is right exact because $\mG$ is formally smooth over $K^+$, so for any $k>0$ and any affinoid $(S,S^+)$ in $\Sous_{K,\et}$, the map $\mG(S^+ )\to \mG(S^+/p^k\m)$ is surjective. For $k=0$, we observe that after the passage to an analytic cover, any $x\in \mG(S^+/\m)$ lifts to $\mG(S^+/p^{\epsilon})$ for some $\epsilon>0$ as $\mG$ is locally of topologically finite presentation.
\end{proof}

We can describe $G_k$ in more classical terms as follows: let $\wh{\mathcal G}$ be the completion of $\mathcal G$ at the origin. This is a formal Lie group: any local choice of generators of the sheaf of ideals defining the unit section in $G$ induces an isomorphism of formal schemes
\[\wh{\mathcal G}\cong \Spf(K^+[[T_1,\dots,T_d]]);\]
see \cite[Chapter II,  pp.~25--26]{MessingCrystals}. 
In particular, on global sections, $\wh{\mathcal G}$ defines a formal group law over $K^+$.

\begin{Lemma}\label{l:explicit-description-of-G_k}\leavevmode
	\begin{enumerate}
		\item $G_0$ is represented by the adic generic fibre  $\wh{\mG}^{\ad}_{\eta}$ of\, $\wh{\mG}$.
		\item\label{l:edGk-2} The natural morphism $G_k\to G_0$ is isomorphic to the adic generic fibre of the morphism of affine formal schemes defined on global sections by
		\[ K^+[[T_1,\dots,T_d]]\lra K^+[[T'_1,\dots,T'_d]],\quad T_i\longmapsto p^kT'_i.\]
	\end{enumerate}
	In particular, $G_k\subseteq G$ is represented by a normal rigid open subgroup whose underlying rigid space is isomorphic to an open ball.
\end{Lemma}

\begin{proof}
	Let $\Spf(A)\subseteq \mG$ be any open neighbourhood of the unit section, and let $I\subseteq A$ be the ideal of the unit section. Then the functor $\wh{\mG}^{\ad}_{\eta}$ is the sheafification of the functor that sends  $(R,R^+)$ to the continuous homomorphisms $\varphi\colon A\to \varprojlim A/I^n=K^+[[T_1,\dots,T_d]]\to R^+$. By continuity, the images of the $T_i$ land in the topologically nilpotent elements $\m R^+$. Thus $\varphi$ reduces mod $\m$ to the unit section $A\to K^+/\m \to R^+/\m$. Conversely, if $A\to R^+$ factors mod $\m$ through the unit section, it already does so mod $p^{\epsilon}$ for some $\epsilon>0$ since $A$ is of topologically finite presentation. Thus for all $n\in \N$, the map $A\to R^+\to R^+/p^n$ factors through $A/I^m$ for some $m$, and in the limit we get the desired map.
	
	For part~\eqref{l:edGk-2}, we similarly observe that if a map $A\to R^+$ factors through $K^+[[T'_1,\dots,T'_d]]$, then it factors through the unit section after reduction mod $p^k\m$. Conversely, if it factors through the unit section mod $p^k\m$, then the induced morphism $K^+[[T_1,\dots,T_d]]\to R^+$ from the first part must send $T_i$ into $p^k\m R^+$ and thus factors through the displayed map.
	
	That $G_k$ is normal can be proved on the level of sheaves, where it follows from the short exact sequence in \cref{l:ses-of-G_k}.
\end{proof}

Combined with our work in the previous sections, this shows the following. 

\begin{Proposition}\label{l:cohomo-comp-G^+/G^+_c}
	The sheaf\, $\overline G_{k}$ on $\Sous_{K,\et}$ is already a $v$-sheaf. Moreover,
	\[ R^1\nu_{\ast}\overline G_k=1.\]
	If\, $G$ is commutative,  we more generally have $R\nu_{\ast}\overline G_{k}=\overline G_{k}$.
\end{Proposition}

\begin{proof}
	By Lemmas~\ref{l:ses-of-G_k} and~\ref{l:explicit-description-of-G_k}, we know that $\overline G_k=G/G_k$ is the quotient of the rigid group $G$ by the normal open subgroup $G_k$. By \cref{p:approx-property-for-G/U}, it follows that $\overline G_k$ is a $v$-sheaf with the desired properties.
\end{proof}

This allows us to pass from adic spaces to diamonds in the following: the rigid groups $G$ and $G_k$ represent sheaves on the site $\Dmd_{K,v}$ of diamonds over $K$ with the $v$-topology.
Due to \Cref{l:cohomo-comp-G^+/G^+_c} and \Cref{l:ses-of-G_k}, we may also extend $\overline{G}_k$ to a $v$-sheaf on the site $\Dmd_{K,v}$ with the $v$-topology by setting $\overline{G}_k:=G/G_k$.

The system $(G_k)_{k\in \N}$ forms a neighbourhood basis of  open subgroups and gives us a notion of completeness of $G$, generalising the isomorphism $\GL_n(\O^+)=\varprojlim_k \GL_n(\O^+/p^k\m)$. 

\begin{Lemma}\label{l:G^+-complete}
	The natural map $ G\to\varprojlim_{k\in \N} \overline{G}_{k}$
	is an isomorphism of $v$-sheaves.
\end{Lemma}

\begin{proof}
	It is clear from \cref{l:explicit-description-of-G_k} that $\varprojlim_{k\in \N} G_k=\cap_{k\in \N} G_k=1$, so the map is injective. To see that it is surjective, suppose we are given a compatible system of elements $x_k\in G/G_k(R,R^+)$ for some perfectoid $(R,R^+)$. Using that the $v$-site is replete, we can inductively find a $v$-cover $(R,R^+)\to (S,S^+)$ such that the $x_k$ are in the image of $\mG(S^+/p^k\m)\to G/G_k(S,S^+)$ and form a compatible system of elements in $\mG(S^+/p^k\m)$. On the level of formal schemes, this defines an element $x\in\mG(S^+)$ whose image in $G(S,S^+)$ is a preimage of $(x_k)_{k\in \N}$.
\end{proof}

\begin{Remark}
	Due to Lemma~\ref{l:G^+-complete}, one could now generalise Faltings' notion of generalised representations by defining ``generalised $G$-representations'' to be compatible systems of $\overline G_{k}$-torsors on the \'etale site. Due to \cref{l:cohomo-comp-G^+/G^+_c} and \cref{l:G^+-complete}, one then sees exactly like in \cref{s:generalised-representations} that these are equivalent to $G$-torsors on $X_v$.
\end{Remark}

We can now identify  the subgroups that appeared in the context of the exponential with the integral subgroups $G_k$ of this section in the context of good reduction. 

\begin{Definition}\label{d:mg-plus_k}
	For any $0\leq k \in \log|K|$,
	we denote by $\mathfrak g_k^+\subseteq \mg$ the subsheaf $p^{k}\m \mg^+$, represented by an open rigid subgroup of $\mathfrak g$. We then set $\overline{\mg}^+_k:=\mg^+/\mg^{+}_k$ on $\Dmd_{K,v}$, the $v$-site of diamonds. Since $\mg^+$ is free, this is isomorphic as an $\O^+$-module to $(\O^+/p^{k}\m\O^+)^d$, where $d=\dim G$.
\end{Definition}

\begin{Lemma}\label{l:exp-vs-integral-subgroups}
	There is an $\alpha\geq 1/(p-1)$ such that for any $\alpha<k\in \log|K|$, the exponential of\, $G$ is defined on $\mg^+_k$ and restricts to an isomorphism of rigid spaces
	\[ \exp\colon\mg^+_k \longisomarrow G_k.\]
\end{Lemma}

\begin{proof}
	By \Cref{p:general-exponential},  the exponential $\mg^\circ\to G$ induces the identity on Lie algebras. It therefore corresponds to an isomorphism from $\mg^+_0$ to $G_0$ in the localised category of 
	\Cref{t:Lie-grp-Lie-alg-correspondence}.
	Its morphism of Lie algebras being the identity means that the associated morphism of formal Lie groups over $K$ is therefore of the form $F\colon K[[T_1,\dots,T_n]]\to K[[T_1,\dots,T_n]]$ with $F(T_i)=T_i+ (\text{terms of higher degree})$. As $F$ converges on some disc centred at $0$, this shows that after $T_i$ is replaced by $T_i'$ on both sides via $T_i\mapsto p^{k}T_i'$ for some $k\in \log|K|$, $k\gg 0$, the morphism $F$ restricts to $F\colon K^+[[T'_1,\dots,T'_n]]\to K^+[[T'_1,\dots,T'_n]]$. Passing to the adic generic fibre, we see from  \cref{l:explicit-description-of-G_k}\eqref{l:edGk-2} that the left-hand side becomes $\mg^+_k$ and the right-hand side becomes~$G_k$. Explicitly, we now define $\alpha$ to be the infimum of all $k>1/(p-1)$ with this property.
\end{proof}

\begin{Definition}
	Recall that  $\alpha_0:=1/(p-1)$ if $p>2$ and $\alpha_0=1/4$ otherwise.
	We denote by $\alpha$ the infimum of all $k\geq \alpha_0$ for which  \cref{l:exp-vs-integral-subgroups} holds. As before, we denote by $\log$ the inverse of $\exp$.
\end{Definition}

\begin{Lemma}\label{l:cong-preserves-mgk}
	For any $k> \alpha$, the adjoint action of\, $G_k$ preserves $\mg_k^+$.
\end{Lemma}

\begin{proof}
	This follows from Lemmas~\ref{l:exp-vs-integral-subgroups} and~\ref{l:exp-on-matrix}.
\end{proof}

\begin{Lemma}\label{l:lifting-ses-for-general-G}
	For any $\alpha<r<s\in \Q$ with $s\leq 2r-\alpha_0\in \Q$, the exponential induces an isomorphism of abelian sheaves on $\Dmd_{K,v}$
	\[ \exp\colon\mg_{r}^+/\mg_{s}^+\longisomarrow G_{r}/G_{s}.\]
	In fact, we already have an isomorphism $\exp\colon \mg_{r}^+(X)/\mg_{s}^+(X)\isomarrow G_{r}(X)/G_{s}(X)$ for any $X\in \Sous_{K}$.  We thus get a short exact sequence on $\Dmd_{K,v}$ $($in fact, already on $\Sous_{K,\et})$
	\[ 0\lra \overline\mg^+_{s-r}\xrightarrow{\exp} \overline{G}_{s}\lra \overline{G}_{r}\lra 1.\]
\end{Lemma}

\begin{Remark}
	For $G=\GL_n$, this is the natural isomorphism
	\[ M_n(p^r\m\O^+/p^{s}\m)\longisomarrow 1+p^r\m M_n(\O^+)/1+p^{s}\m M_n(\O^+),\quad x\longmapsto 1+x\]
	which coincides with the exponential because the conditions on $r$ and $s$ ensure that 
	\[x^2/2!+x^3/3!+\dots \in p^{s}\O^+\quad \text{for any }x\in p^r\O^+\]
	by the usual estimate $v_p(x^n/n!)> nv_p(x)-\tfrac{n-1}{p-1}$. The analogue holds for linear algebraic groups.
\end{Remark}

\begin{proof}
  By the commutative diagram in \cref{p:general-exponential}, $\BCH$ defines on $\mathfrak g^+_{r}$ a rigid group structure isomorphic to $G_{r}$. It suffice to prove that this agrees with the additive group structure on the quotient $\mathfrak g^+_{r}/\mathfrak g^+_{s}$. This is guaranteed by the following lemma.
\end{proof}

\begin{Lemma}\label{l:estimate-on-BCH}
		For any $r> \alpha$, let $x,y\in \mathfrak g^+_{r}$. Then $\BCH(x,y)$ converges in $\mathfrak g^+_{r}$, and
		\[\BCH(x,y)\equiv x+y\bmod \mathfrak g^+_s\] 
		for any $r< s<2r-\alpha_0$. 
	\end{Lemma}

\begin{proof}
		This follows from the estimates in \cite[Equation~(26) and Proposition~17.6]{Schneider_Lie}: we have $\BCH=\sum_{n\geq 1} H_n$, where $H_n(X,Y)$ is homogeneous of degree $n$, and 
		$\|H_n\|\leq |p|^{-(n-1)\alpha_0}$. Let us  define a valuation $v_p$ on $\mg$ via $v_p(z):=\sup\{\log_{|p|}|a|:a\in K \text{ s.t. }z\in a\cdot \mathfrak g^+\}$ for $z\in \mathfrak g$. Then since $x,y\in \mathfrak g_r^+$ satisfy $v_p(x),v_p(y)>r$,
		it follows that for $n\geq 2$, we have
		\[v_p(H_n(x,y))> nr-(n-1)\alpha_0= n(r-\alpha_0)+\alpha_0\geq 2r-\alpha_0\geq  s.\]
	Thus 
		\[ \BCH(x,y)=x+y+H_2(x,y)+\dots\in x+y+p^s\m\mathfrak g^+\]
		is of the desired form.
	\end{proof}

\subsection{$\boldsymbol{G}$-torsors on perfectoid spaces}
Given our technical preparations, we can now generalise the results of \cref{s:generalised-representations}. First,
by a generalisation of \cref{l:local-freeness--in-the-limit-in-the-v-topology}, small $G$-torsors on affinoid perfectoid $Z$ are trivial. 

\begin{Lemma}\label{l:G^+-torsors-on-perfectoids}
	
	Let $X$ be any diamond over $K$ with $H^1_v(X,\m\O^+/p^\epsilon\m)=0$ for any $\epsilon>0$, for example any affinoid perfectoid space, and let $k>\alpha$. Then a $G$-torsor $V$ on $X_v$ is trivial if and only if the associated $\overline{G}_{k}$-torsor $V_k$ on $X_v$ is trivial. In particular,
	\[ H^1_v(X,G_k)=1.\]
\end{Lemma}

\begin{proof}
	We set $r:=k>\alpha>\alpha_0$ and let $s\in \mathbb R$ be such that $r<s<2r-\alpha_0$.
	We then apply $H^1_v(X,-)$ to the short exact sequence of $v$-sheaves in Lemma~\ref{l:lifting-ses-for-general-G}. Since $\overline\mg_{s-r}^+=(\m\O^+/p^{s-r}\m)^{d}$, it follows from the assumption that $H^1_v(X,\overline\mg_{s-r}^+)=0$. This shows that if $V_k$ is trivial, then so is $V_s$. Replacing $k$ by $s$ and arguing inductively, this shows that $V_s$ is trivial for any $s>k$. The same long exact sequence shows inductively that any element in $\overline{G}_{k}(X)$ can be lifted to $\overline{G}_{s}(X)$ for any $s>k$.
	
	We then use that we have a Milnor exact sequence of pointed sets
	\[ 1\lra R^1\lim_s \overline{G}_{s}(X)\lra H^1_v(X,G)\lra \varprojlim_s H^1_v(X,\overline{G}_{s}).\]
	Since we have just seen that $(\overline{G}_{s}(X))_{s\in \N}$ is a Mittag-Leffler system, the first term vanishes. This shows that an element of $H^1_v(X,G)$ vanishes if and only if its image in $\varprojlim_s H^1_v(X,\overline{G}_{s})$ vanishes, which holds if and only if its image in $H^1_v(X,\overline{G}_{k})$ vanishes. This shows the first part.
	
	By \cref{l:G^+-complete}, we can also deduce that the map $G(X)\to \overline{G}_{k}(X)$ is surjective.
	The last sentence then follows from the long exact sequence of $H^0_{v}(X,-)$  associated to the short exact sequence of \cref{l:ses-of-G_k}.
\end{proof}

We can now prove the main theorem of this article, generalising Kedlaya--Liu's \cref{t:loc-free-O-modules} and
\cref{t:loc-free-O^+-modules}, which were the cases of $G=\GL_n$ and $G=\GL_n(\O^+)$. 

\begin{Theorem}\label{t:G-torsors-on-perfectoids}
	Let $X$ be a perfectoid space, and let $G$ be any rigid group over $K$. Then the categories of\, $G$-torsors on $X_{\et}$ and $G$-torsors on $X_{v}$ are equivalent.
\end{Theorem}

\begin{proof}
	By \cref{c:neighbourhood-basis-good-reduction}, there exists a rigid open subgroup $U\subseteq G$ which has good reduction. By \cref{p:reduction-of-structure-group-to-open-subgroup}, the map
	$R^1\nu_{\ast}U\to R^1\nu_{\ast}G$
	is surjective, so it suffices to prove that $R^1\nu_{\ast}U=1$. 
	We have thus reduced to the case that $G$ has good reduction. By the same argument, we may then further replace $G$ by the open subgroup $G_k$ for $k>\alpha$. Then by \cref{l:G^+-torsors-on-perfectoids}, we have $R^1\nu_{\ast}G_k=1$.
\end{proof}

\begin{Corollary}\label{c:G-torsor-is-diamond}
	Let $X$ be any adic space over $K$. Then the categories of\, $G$-torsors on $X_{\qproet}$ and  $X_{v}$ are equivalent. Moreover, any geometric $G$-torsor on $X_v$ is a diamond.
\end{Corollary}

\begin{proof}
	The first part follows since by \cref{l:proet-Colmez-cover}, the space $X$ admits a quasi-pro-\'etale cover by a perfectoid space. For the second part,
	 let $E$ be a geometric $G$-torsor on $X_v$. The statement is \'etale-local on $X$, so we may assume that $X$ has a pro-finite-\'etale perfectoid Galois cover $X'\to X$ such that $E\times_XX'\cong G\times X'$. Then the projection $G\times X'\to E$ is a pro-finite-\'etale cover of $E$ by a sousperfectoid space. Hence $E$ is a diamond. 
\end{proof}


\newcommand{\etalchar}[1]{$^{#1}$}


\begin{thebibliography}{BGH{\etalchar{+}}22+++}
	
\bibitem[AR94]{Adamek-Rosicky}
J.~Ad\'{a}mek and  J.~Rosick\'{y}, 
{\em Locally presentable and accessible categories}, London Math.\ Soc.\ Lecture Note Ser., vol.~189, Cambridge Univ.\ Press, Cambridge, 1994, \doi{10.1017/CBO9780511600579}.

\bibitem[BS15]{bhatt-scholze-proetale}
B.~Bhatt and  P.~Scholze, 
{\em {T}he pro-\'{e}tale topology for schemes}, Ast\'{e}risque \textbf{369} (2015), 99--201, \doi{10.24033/ast.960}.
        
\bibitem[BGH{\etalchar{+}}22]{perfectoid-covers-Arizona}
C.~Blakestad, D.~Gvirtz, B.~Heuer, D.~Shchedrina, K.~Shimizu, P.~Wear and  Z.~Yao, 
{\em {P}erfectoid covers of abelian varieties}, Math.\ Res.\ Lett.\ \textbf{29} (2022), no.~3, 
631--662,  \doi{10.4310/MRL.2022.v29.n3.a2}.
	
\bibitem[Bou72]{Bourbaki-Lie}
N.~Bourbaki, 
{\em \'{E}l\'{e}ments de math\'{e}matique. {F}asc. {XXXVII}.  {G}roupes et alg\`ebres de {L}ie. {C}hapitre {II}: {A}lg\`ebres de {L}ie libres. {C}hapitre {III}: {G}roupes de {L}ie}, Actualit\'{e}s Sci.\ Indust.\, No.~1349. Hermann, Paris, 1972.
	

\bibitem[Col02]{Colmez-espacesdeBanach}
P.~Colmez, 
{\em {E}spaces de {B}anach de dimension finie}, J.~Inst.\ Math.\ Jussieu \textbf{1} (2002), no.~3, 331--439, \doi{10.1017/S1474748002000099}.

\bibitem[SGA3]{SGA3}
M.~Demazure and A.~Grothendieck (eds), \emph{Sch\'emas en groupes. I}, S\'eminaire de G\'eom\'etrie Alg\'ebrique du Bois Marie 1962/64 (SGA 3), Lecture Notes in Math., vol.~151, Springer-Verlag, Berlin-New York, 1970.

\bibitem[Fal05]{Faltings_SimpsonI}
G.~Faltings, 
{\em {A} {$p$}-adic {S}impson correspondence}, Adv.\ Math., \textbf{198} (2005), no.~2, 847--862, \doi{10.1016/j.aim.2005.05.026}.
	
\bibitem[Far19]{Fargues-groupes-analytiques}
L.~Fargues, 
{\em {G}roupes analytiques rigides {$p$}-divisibles}, Math.\ Ann.\ \textbf{374} (2018), no.~1-2, 723--791, \doi{10.1007/s00208-018-1782-9}.
	
\bibitem[GR03]{GabberRamero}
O.~Gabber and L.~Ramero, 
{\em {A}lmost ring theory},  Lecture Notes in Math., vol.~1800, Springer-Verlag, Berlin, 2003, \doi{10.1007/b10047}.
	
\bibitem[Gro68]{Grothendieck_BrauerIII}
A.~Grothendieck, 
{\em Le groupe de {B}rauer. {III}. {E}xemples et compl\'{e}ments}, in: {\em Dix expos\'{e}s sur la cohomologie des sch\'{e}mas}, pp.~88--188, Adv.\ Stud.\ Pure Math., vol.~3, North-Holland Publishing Co., Amsterdam, 1968.

\bibitem[HK25]{HK_sheafiness}
D.~Hansen and  K.\,S.~Kedlaya.
{\em Sheafiness criteria for {H}uber rings}, preprint (2025), available at \url{https://kskedlaya.org/papers/criteria.pdf}.

\bibitem[Heu21a]{heuer-diamantine-Picard}
B.~Heuer, 
{\em {D}iamantine {P}icard functors of rigid spaces}, preprint \arXiv{2103.16557} (2021).
	
\bibitem[Heu21b]{heuer-Picard-good-reduction}
\bysame, 
{\em Line bundles on perfectoid covers: case of good reduction}, preprint \arXiv{2105.05230} (2021).
	
\bibitem[Heu22a]{heuer-v_lb_rigid}
\bysame, 
{\em Line bundles on rigid spaces in the $v$-topology}, Forum of Math.\ Sigma \textbf{10} (2022), Paper No.~e82, \doi{10.1017/fms.2022.72}.
	
\bibitem[Heu22b]{heuer-sheafified-paCS}
\bysame, 
{\em Moduli spaces in $p$-adic non-abelian {H}odge theory}, preprint \arXiv{2207.13819} (2022).
	
\bibitem[Hub94]{Huber-ageneralisation}
R.~Huber, 
{\em {A} generalization of formal schemes and rigid analytic varieties}, Math.~Z.\ \textbf{217} (1994), no.~4, 513--551, \doi{10.1007/BF02571959}.
	
\bibitem[Hub96]{huber2013etale}
\bysame, 
{\em \'{E}tale cohomology of rigid analytic varieties and adic spaces}, Aspects Math., vol.~E30, Friedr.\ Vieweg \& Sohn, Braunschweig, 1996, \doi{10.1007/978-3-663-09991-8}.
	
\bibitem[Ill71]{IllusieCotangent}
L.~Illusie.
{\em Complexe cotangent et d\'{e}formations. {I}}, Lecture Notes in Math., vol.~239, Springer-Verlag, Berlin-New York, 1971, \doi{10.1007/BFb0059052}.
	
\bibitem[KL15]{KedlayaLiu-rel-p-p-adic-Hodge-I}
K.\,S.~Kedlaya and  R.~Liu.
{\em {R}elative {$p$}-adic {H}odge theory: foundations}, Ast\'{e}risque \textbf{371} (2015), \doi{10.24033/ast.957}.

\bibitem[KL16]{KedlayaLiu-II}
\bysame, 
{\em {R}elative $p$-adic {H}odge theory, {I}{I}: {I}mperfect period rings}, preprint \arXiv{1602.06899} (2016).

\bibitem[LZ17]{LiuZhu_RiemannHilbert}
R.~Liu and  X.~Zhu, 
{\em Rigidity and a {R}iemann--{H}ilbert correspondence for {$p$}-adic local systems}, Invent.\ Math.\ \textbf{207} (2017), no.~1, 291--343, \doi{10.1007/s00222-016-0671-7}.

\bibitem[L{\"{u}}t95]{Lutkebohmert_structure_of_bounded}
W.~L{\"{u}}tkebohmert, 
{\em The structure of proper rigid groups}, J.~reine angew.\ Math.\ \textbf{468} (1995), 167--219, \doi{10.1515/crll.1995.468.167}.

\bibitem[MW22]{MannWerner_LocSys_p-adVB}
L.~Mann and  A.~Werner, 
{\em {L}ocal {S}ystems on {D}iamonds and p-{A}dic {V}ector {B}undles}.  Int.\ Math.\ Res.\ Not., IMRN \textbf{2023}, no.~15, 12785--12850, \doi{10.1093/imrn/rnac182}.

\bibitem[Mes72]{MessingCrystals}
W.~Messing, 
{\em {T}he crystals associated to {B}arsotti-{T}ate groups: with applications to abelian schemes}.  Lecture Notes in Math., vol.~264. Springer-Verlag, Berlin-New York, 1972, \doi{10.1007/BFb0058301}.
	
\bibitem[MT21]{MorrowTsuji}
M.~Morrow and T.~Tsuji.
{\em Generalised representations as q-connections in integral $p$-adic {H}odge theory}, preprint \arXiv{2010.04059} (2021).
		
\bibitem[Sch11]{Schneider_Lie}
P.~Schneider, 
{\em {$p$}-adic {L}ie groups}, Grundlehren math.\ Wiss., vol.~344, Springer, Heidelberg, 2011, \doi{10.1007/978-3-642-21147-8}.
	
\bibitem[Sch12]{perfectoid-spaces}
P.~Scholze, 
{\em {P}erfectoid spaces}, Publ.\ Math.\ Inst.\ Hautes \'{E}tudes Sci.\ \textbf{116} (2012), 245--313, \doi{10.1007/s10240-012-0042-x}.
	
\bibitem[Sch13a]{Scholze_p-adicHodgeForRigid}
\bysame, 
       {\em {$p$}-adic {H}odge theory for rigid-analytic varieties}, Forum Math.\ Pi \textbf{1} (2013), Paper No.~e1,  \doi{10.1017/fmp.2013.1}.
	
\bibitem[Sch13b]{Scholze2012Survey}
\bysame, 
{\em {P}erfectoid spaces: a survey}, in: {\em Current developments in mathematics 2012}, pp.~193--227.  Int.\ Press, Somerville, MA, 2013, \doi{10.4310/CDM.2012.v2012.n1.a4}.
	
\bibitem[Sch22]{etale-cohomology-of-diamonds}
\bysame, 
{\em {\'E}tale cohomology of diamonds}, preprint \arXiv{1709.07343} (2022).

\bibitem[SW13]{ScholzeWeinstein}
P.~Scholze and  J.~Weinstein, 
{\em {M}oduli of {$p$}-divisible groups}, Camb.\ J.~Math.\ \textbf{1} (2013), no.~2, 145--237, \doi{10.4310/CJM.2013.v1.n2.a1}.
	
\bibitem[SW20]{ScholzeBerkeleyLectureNotes}
\bysame, 
{\em {B}erkeley {L}ectures on $p$-adic {G}eometry}, Ann.\ of Math.\ Stud., vol.~207, Princeton Univ.\ Press, Princeton, NJ, 2020.
       
\bibitem[Ser06]{Serre_Lie}
J.-P. Serre, 
{\em Lie algebras and {L}ie groups} (corrected fifth printing of the second (1992) edition), Lecture Notes in Math., vol.~1500, Springer-Verlag, Berlin, 2006, \doi{10.1007/978-3-540-70634-2}.

\bibitem[Sta23]{StacksProject}
The Stacks Project Authors, \emph{Stacks project}, \url{https://stacks.math.columbia.edu}, 2023. 
	
\bibitem[W{\"u}r23]{wuerthen_vb_on_rigid_var}
M.~W\"{u}rthen, 
{\em {V}ector {B}undles with {N}umerically {F}lat {R}eduction on {R}igid {A}nalytic {V}arieties and $p$-{A}dic {L}ocal {S}ystems}, Int.\ Math.\ Res.\ Not.\  IMRN \textbf{2023}, no.~5, 4004--4045, \doi{10.1093/imrn/rnab214}.
	
\bibitem[Xu17]{XuTransport_parallele}
D.~Xu, 
{\em Transport parall\`ele et correspondance de {S}impson {$p$}-adique}, Forum Math.\ Sigma,  \textbf{5} (2017), Paper No.~e13, \doi{10.1017/fms.2017.7}.
	
\bibitem[YZ21]{YangZuo-padicSimpson}
J.~Yang and  K.~Zuo,
{\em A note on the $p$-adic {S}impson correspondence}, preprint \arXiv{2012.02058} (2021).
	
\bibitem[Zav21]{Zavyalov_acoh}
B.~Zavyalov, 
{\em Almost coherent modules and almost coherent sheaves}, preprint \arXiv{2110.10773} (2021).
	
\end{thebibliography}
\end{document}